\documentclass[12pt]{article}
\usepackage{mathtools}
\usepackage{amsmath}
\usepackage{amsfonts}
\usepackage{latexsym}
\usepackage{amssymb}
\usepackage{stmaryrd}
\usepackage{overpic}
\usepackage{graphicx}
\usepackage{multirow}
\usepackage{mathrsfs} 
\usepackage{float}  
\usepackage{color}
\usepackage{subcaption}
\newtheorem{theorem}{Theorem}[section]
\newtheorem{assumption}[theorem]{Assumption}
\newtheorem{corollary}[theorem]{Corollary}
\newtheorem{definition}[theorem]{Definition}
\newtheorem{example}[theorem]{Example}
\newtheorem{lemma}[theorem]{Lemma}

\newtheorem{proposition}[theorem]{Proposition}

\newenvironment{proof}[1][Proof]{\textbf{#1.} }
{\ \rule{0.75em}{0.75em}\smallskip}

\newcommand{\bx}{\boldsymbol{x}}
\newcommand{\by}{\boldsymbol{y}}

\begin{document}

\begin{center}
\large\bf Numerical Analysis of Stochastic Elliptic Variational Inequalities of the First Kind
\end{center}

\begin{center}
Chenhui Zhu\footnote{School of Mathematics and Statistics, Xi'an Jiaotong University,
Xi'an, Shaanxi 710049, P. R. China. Email: {chzhu@stu.xjtu.edu.cn}},\quad
Fei Wang\footnote{School of Mathematics and Statistics, Xi'an Jiaotong University,
Xi'an, Shaanxi 710049, P. R. China. The work of this author was partially
supported by the National Natural Science Foundation of China (Grant No.\ 12171383). Email: feiwang.xjtu@xjtu.edu.cn}, \quad
Weimin Han\footnote{Department of Mathematics, University of Iowa, Iowa City, IA 52242, USA. The work of this author was partially supported by Simons Foundation Collaboration Grants (Grant No.\ 850737). Email: weimin-han@uiowa.edu}

\end{center}
\bigskip
\begin{quote}
{\bf Abstract.} 
{This paper presents a numerical approach to the stochastic obstacle problem using the stochastic Galerkin (SG) method. Due to the low regularity of the solution, linear finite elements are employed in both the physical and random variable spaces.} Properties of random fields and variational inequalities of the first kind are employed to establish the well-posedness of the problem. Finite element spaces are introduced to construct suitable approximation subspaces, and a comprehensive SG formulation is proposed to solve the stochastic obstacle problem. Well-posedness of the discrete formulation is shown and an optimal error estimate for the numerical solution in the $H^1$-norm is derived. Numerical experiments validate the effectiveness of the SG method, showing that both the expectation error and second moment error converge at a rate of $O(h)$ in the $H^1$-norm, consistent with theoretical predictions.



\end{quote}

\section{Introduction}

Variational inequalities (VIs) are a fundamental class of mathematical problems that extend the concepts of optimization and equilibrium. They are extensively applied in fields such as applied mathematics, physics, engineering, and economics to model systems involving constraints, inequalities, or non-smooth behavior (\cite{DL1976,KO1988,HR2013,Kinderlehrer2000,HS2002,EJK2005}). VIs offer a robust framework for addressing challenges in the presence of complex constraints or irregularities. {Since closed-form solutions are rarely available for VIs arising in practical applications, numerical methods are required to solve such problems. Various numerical approaches have been developed and studied to this end, including, notably, finite element methods~\cite{falk1974error,TLG1981,Glowinski1984,Hlavacek1988}, discontinuous Galerkin methods~\cite{wang2010discontinuous, wang2014discontinuous}, and virtual element methods~\cite{wang2018friction, wang2020obstacle, feng2019virtual}.}

In many complex physical and engineering models, significant uncertainties arise from various sources, including parameters, coefficients, forcing terms, and boundary conditions. These models are often described using stochastic partial differential equations (SPDEs). In recent years, significant attention has been given to the stochastic Galerkin method (\cite{Babu2001, Babuska2002,Babuska2004,Babuska2005,Shen2019}) and stochastic collocation methods (\cite{Babuska2007,Xiu2007}) as effective tools for uncertainty quantification and propagation in random PDEs.

The stochastic Galerkin method addresses uncertainty by expanding the solution in a polynomial basis with respect to random variables, thereby transforming the problem into a high-dimensional deterministic system. {Galerkin finite element techniques, including both $h$- and  $p$-versions~\cite{Babu2001, Babuska2005, Babuska2004}, have been effectively used to approximate statistical quantities of interest. In contrast, the stochastic collocation (SC) method (\cite{Babuska2007,Xiu2007}) offers a non-intrusive alternative based on evaluating the solution at carefully chosen points in the stochastic space and interpolating the result.}

Building on these methods, several theoretical investigations have been conducted to support their applicability in various contexts. For example,~\cite{Shen2019} presents a detailed analysis of the stochastic Galerkin method for optimal control problems governed by elliptic partial differential equations with random coefficients and state constraints. In more challenging settings, such as stochastic obstacle problems, the convergence behavior may degrade due to the limited regularity of the solution with respect to the random parameters. Consequently, polynomial chaos discretizations fail to achieve exponential convergence rates in the stochastic domain, thereby diminishing the efficiency of SG/SC methods. Nonetheless, if the coincidence set is only mildly sensitive to stochastic variations, satisfactory convergence behavior may still be observed. For instance,~\cite{Forster2010} reports promising numerical results for elliptic variational inequalities with random diffusion coefficients using a stochastic collocation method; however, a rigorous theoretical justification for the observed performance remains lacking.

One of the main challenges in numerically solving stochastic obstacle problems using stochastic collocation methods arises from the limited global regularity of the solution with respect to the random parameters. This lack of smoothness is a direct consequence of the complementarity condition, which introduces non-smooth dependence on the stochastic variables. {In contrast, Monte Carlo (MC) methods, which are widely used in statistical simulations, offer a flexible and non-intrusive approach that remains applicable even in the presence of low regularity.} For example, multilevel Monte Carlo finite element methods (MLMC-FEMs) have been applied to stochastic elliptic variational inequalities in~\cite{Kornhuber2014}, and a convergence theory for multilevel sample variance estimators was developed in~\cite{Bierig2015}. The overall error of such MC-based finite element methods is bounded by \(O(h)+O(M^{-1/2})\), where \(M\) denotes the number of samples. 

However, for stochastic obstacle problems, each sample requires the solution of a nonlinear variational inequality, which is significantly more expensive than solving a linear PDE. Consequently, the cost per sample in MC methods becomes substantially higher, making the slow convergence rate with respect to \(M\) even more problematic when high accuracy is desired. To overcome this bottleneck we adopt stochastic Galerkin finite element discretization, which replaces many independent solves by a single coupled deterministic variational inequality in a tensor-product space (\cite{Babu2001}). Once the SG system is solved, statistical quantities of the solution such as the expectation and the second moment are obtained by inexpensive post-processing of the SG coefficients. To the best of our knowledge, a rigorous numerical analysis of SG finite element discretizations for stochastic obstacle problems has been largely absent. This work fills that gap by establishing the well-posedness of both the continuous problem and its SG discretization, and by deriving a priori error estimates.

This paper addresses the stochastic obstacle problem using finite element methods, with a focus on theoretical analysis and numerical approximation. Due to the limited regularity of the solution, our framework is restricted to linear finite elements in both the spatial and stochastic domains. The structure of the paper is as follows. In Section \ref{sec2}, we introduce the stochastic obstacle problem in the form of a variational inequality of the first kind and explore its well-posedness. In Section \ref{sec3}, we construct finite element spaces, employing piecewise linear basis functions for both the physical and stochastic discretizations. Based on these spaces, in Section \ref{sec4}, we present the full finite element formulation, including the Galerkin approximation and providing a priori error analysis. Finally, in Section \ref{sec5}, we illustrate the performance of the proposed method through numerical experiments that confirm the theoretical convergence behavior across different stochastic configurations.

\section{Obstacle Problem with Random Coefficients}\label{sec2}
\setcounter{equation}0
\subsection{{Deterministic Obstacle Problem}}

First, we describe a deterministic obstacle problem (\cite{Ro1987,Atkinson2009Theoretical}). Let $D\subset \mathbb{R}^n$ be an open, bounded Lipschitz domain, the positive integer $n$ being the spatial dimension.
{The weak formulation of the obstacle problem is a variational inequality}: Find $u \in K$ such that 
\begin{align}
    \label{eq1}
 \int_{D} a \nabla u \cdot \nabla (v-u) \, \mathrm{d} \bx \geq \int_{D} f(v-u) \, \mathrm{d} \bx \quad \forall v \in K,
\end{align}
where \( f \in L^2(D) \),  \( a \in L^{\infty}(D) \)   and \( K \) is the set of admissible displacements, 
\begin{equation*}
    K = \{v \in H_0^1(D) \left| \right. v \geq g \text{ a.e. in } D\},
\end{equation*}
in which $g$ denotes the height function of the obstacle.
In the space $H^1_0(D)$, we use its canonical inner product
\[ (u,v)_{H^1_0(D)}=\int_D \nabla u\cdot\nabla v\, \mathrm{d}\bx \]
and the corresponding norm $\|v\|_{H^1_0(D)}=(v,v)_{H^1_0(D)}^{1/2}$.
Assuming the coefficient $a(\cdot)$ is bounded below by a constant $a_0 > 0$, we can conclude that the
obstacle problem has a unique solution by applying the following result
(\cite[page 3]{Glowinski1984}, \cite[Theorem 11.3.9]{Atkinson2009Theoretical}).

\begin{lemma}	\label{lem1.1}
Let $V$ be a real Hilbert space, and let $K \subset V$ be a non-empty, closed and convex set.
Assume $b(\cdot, \cdot) \colon V \times V \to \mathbb{R}$ is a continuous, $V$-elliptic bilinear form,
$ \ell \in V^\prime$. Then the following elliptic variational inequality
\begin{equation*}
u \in K, \quad	b(u, v-u) \geq \ell (v-u) \quad \forall v \in K
\end{equation*} 
has a unique solution.
\end{lemma}

{While the deterministic formulation assumes precise knowledge of the input data, real-world scenarios often involve uncertainties in the coefficient \(a\), the source term \(f\) and height function \(g\) of the obstacle. These uncertainties will be described by randomness of the data \(a\), \(f\) and \(g\).}

\subsection{{ Stochastic Framework and Functional Setting}}
{We begin by briefly recalling some standard definitions related to random variables and random fields; cf. e.g.,~\cite{Durrett2019Probability,Lord2014introduction,Oksendal2013} for comprehensive introductions.}  Let \((\Omega, \mathcal{F}, \mathbb{P})\) be a probability space, where \(\Omega\) is the set of elementary events, \(\mathcal{F}\) is a \(\sigma\)-algebra rich enough to support all the random variables we consider, and \(\mathbb{P}\) is a probability measure on $\mathcal{F}$. For any \( Y \in L^1(\Omega) \), the expected value is defined by
\[
\mathbb{E}[Y] = \int_{\Omega} Y(\omega) \, \mathrm{d}\mathbb{P}(\omega) = \int_{\mathbb{R}} y \, \mathrm{d}\mu_Y(y),
\]
where \( \mu_Y \) denotes the distribution measure of \( Y \), defined by \( \mu_Y(G) = \mathbb{P}(Y^{-1}(G)) \) for any Borel set \( G \in \mathfrak{B}(\mathbb{R}) \),  and $\mathfrak{B} (\mathbb{R})$ is the Borel $\sigma$-algebra generated by the open subsets of $\mathbb{R}$. If \( \mu_Y \) is absolutely continuous with respect to the Lebesgue measure, then \( Y \) admits a probability density function \( \rho_Y \) such that \( \mathbb{E}[Y] = \int_{\mathbb{R}} y \rho_Y(y) \, \mathrm{d}y \).

Let \( X \) be a Banach space of real-valued functions defined on a domain \( D \subset \mathbb{R}^n \), equipped with a norm \( \|\cdot\|_X \).  
{A real-valued random field \( v\colon D \times \Omega \to \mathbb{R} \) is called an \( X \)-valued random variable if it is measurable with respect to the product \( \sigma \)-algebra \( \mathfrak{B}(D) \otimes \mathcal{F} \).  It belongs to the Bochner space \( L^p(\Omega; X) \), for some \( 1 \leq p < \infty \), if
\[
\|v\|_{L^{p}(\Omega; X)}^{p} \coloneqq \int_{\Omega} \|v(\cdot, \omega)\|_X^p \, \mathrm{d}\mathbb{P}(\omega) < \infty.
\]
For \( p = \infty \), the Banach space \( L^{\infty}(\Omega; X) \) is equipped with the essential supremum norm:
\[
\|v\|_{L^{\infty}(\Omega; X)} \coloneqq \operatorname*{ess\,sup}_{\omega \in \Omega} \|v(\cdot, \omega)\|_X.
\]
In particular, if \( X = L^2(D) \), then \( v \in L^2(\Omega; L^2(D)) \) implies}
\[
\|v\|_{L^2(\Omega; L^2(D))}^{2} = \mathbb{E}\left[ \|v(\cdot, \omega)\|_{L^2(D)}^2 \right] = \int_{\Omega} \int_D |v(\bx, \omega)|^2 \, \mathrm{d}\bx \, \mathrm{d}\mathbb{P}(\omega) < \infty.
\]

Given the distinct structural properties of stochastic functions with respect to \(\omega\) and \(\bx\), their numerical approximations require the use of tensor spaces. 
Let \( H_1 \) be a Hilbert space of functions defined on the physical domain \( D \subset \mathbb{R}^n \), and \( H_2 \) a Hilbert space of functions defined on the parameter domain \( \Gamma \subset \mathbb{R}^M \). The tensor space \(H_1 \otimes H_2\) is defined as the completion of elements of the form \(u(\bx,\by) = \sum_{i=1}^d v_i(\bx) w_i(\by)\), where \(\{v_i\} \subset H_1\) and \(\{w_i\} \subset H_2\), with respect to the norm $\|\cdot\|_{H_1 \otimes H_2}$ induced by the inner product 
\[
(u, \hat{u})_{H_1 \otimes H_2} = \sum_{i,j} (v_i, \hat{v}_ j)_{H_1} (w_i, \hat{w}_j)_{H_2}
\]
for $u(\bx,\by) = \sum_{i=1}^d v_i(\bx) w_i(\by)$ and $\hat{u}(\bx,\by) = \sum_{i=1}^d \hat{v}_i(\bx) \hat{w}_i(\by)$. For domains \(D\) and \(\Gamma\), the tensor space \(H^1(D) \otimes  L^2(\Gamma) \) is equipped with the inner product
\[
(u, \hat{u})_{H^1(D) \otimes L^2(\Gamma)} = \int_{\Gamma} \int_{D} \left( u(\bx,\by) \hat{u}(\bx,\by) + \nabla_{\bx} u(\bx,\by) \cdot \nabla_{\bx} \hat{u}(\bx,\by) \right) \, \mathrm{d}\bx \, \mathrm{d}\by.
\]
Thus, if \(u \in  H^k(D)  \otimes L^2(\Gamma) \), it follows that \(u(\cdot, \by) \in H^k(D)\) for almost every  \( \by \in \Gamma\), and for almost every \(\bx \in D\), \(u(\bx, \cdot) \in L^2(\Gamma)\). Moreover, there exists an isomorphism \( H^k(D) \otimes  L^2(\Gamma)  \simeq L^2(\Gamma; H^k(D)) \simeq H^k(D; L^2(\Gamma)) \) (\cite{Babuska2004}).

With the above notations, we now consider the stochastic obstacle problem: Given \(a \in L^{\infty}\left(\Omega; L^{\infty}(D)\right)\), \( f \in L^2(\Omega; L^2(D))\), \(g \in L^{2}\left(\Omega; H^{1}(D)\right)\), find a function $u$ such that for almost every \( \omega \in \Omega \), \( u(\cdot, \omega) \in K_{\omega} \) and the following inequality holds:
{\small \begin{equation}
	\label{eq2}
	\int_{D} a(\bx, \omega) \nabla u({\boldsymbol x}, \omega) \cdot \nabla (v({\boldsymbol x})-u({\boldsymbol x}, \omega)) \, \mathrm{d}\bx \geq \int_{D} f(\bx, \omega)(v({\boldsymbol x})-u({\boldsymbol x}, \omega)) \, \mathrm{d}\bx \quad\forall\,v \in K_{\omega},
\end{equation}} where \(K_{\omega} = \{v(\cdot, \omega) \in H_0^1(D) \mid v(\cdot, \omega) \geq g(\cdot, \omega) \text{ a.e. in } D\}\). {The gradient operator \(\nabla\) refers to differentiation with respect to $\bx \in D$ unless otherwise stated.} The corresponding pointwise formulation is to find a random function $u(\bx,\omega)$ such that (cf.\ \cite[Example 11.1.1]{Atkinson2009Theoretical})
\begin{align}
	\label{eqe3}
	-\text{div}(a(\bx,\omega)\nabla u(\bx,\omega)) &\geq f(\bx,\omega) &\text{ in } D, \notag \\
	u(\bx,\omega) &\geq g(\bx, \omega)  &\text{ in } D, \notag \\
		(\text{div} (a(\bx,\omega)\nabla u(\bx,\omega))+f(\bx,\omega))(u(\bx,\omega)-g(\bx, \omega)) & = 0 & \text{ in } D,\\
		u(\bx,\omega) &= 0 &\text{ on } \partial D. \notag 
\end{align}

{ To ensure the compatibility of the obstacle with the homogeneous Dirichlet boundary condition, we assume that  
\(g(\bx,\omega) \leq 0 \,\, \text{for a.e. } \bx \in \partial D, \; \omega \in \Omega.\)
}

The existence and uniqueness of $u(\cdot, \omega)$ are ensured by Lemma~\ref{lem1.1} provided that $ 0< a_{\min}(\omega)\leqslant  a(\bx, \omega)\leqslant a_{\max}(\omega)$. To establish well-posedness $\mathbb{P}$-almost surely, we introduce the following assumptions (\cite{Kornhuber2014}):
\begin{assumption}\label{2.1}
	For almost all $\omega\in \Omega$, 
	\begin{align*}
	& a_{\rm min}(\omega) \coloneqq \underset{{\bx} \in D}{\operatorname{ess} \inf} \, a({\bx}, \omega)>0,\\
	& a_{\max }(\omega) \coloneqq \underset{{\bx} \in D}{\operatorname{ess} \sup }\, a({\bx}, \omega)<\infty.
	\end{align*}
\end{assumption}

\begin{assumption}\label{2.2}
There exist two constants  $a_{\min }  $ and $  a_{\max }$ such that 
\[0<a_{\min } \leq a({\bx}, \omega) \leq a_{\max }<\infty, \quad \text { a.e. in } D \times \Omega.\]
\end{assumption}

Obviously,  Assumption~\ref{2.2} is stronger than Assumption~\ref{2.1}.

\subsection{Well-posedness of the stochastic obstacle problem}

Let $ V = L^2(\Omega; H_0^1(D))$, equipped with the inner product \((v,w)_{V} = \mathbb{E}[(\nabla v, \nabla w)_{L^2(D)}]\) and the corresponding norm $\|v\|_V=(v,v)_V^{1/2}$ for \(v,w \in V\). We consider the following stochastic variational inequality (\cite{Bierig2015,Forster2010}):
\begin{equation}
	\label{eq5}
u \in \mathcal{K}, \quad	b(u, v-u) \geq \ell(v-u), \quad \forall v \in \mathcal{K},
\end{equation}
where $ b(\cdot, \cdot)\colon V \times V \to \mathbb{R}$,  $\ell\colon V \to \mathbb{R}$ and $\mathcal{K}$
 are defined as follows:
\begin{align}
b(u, v) & \coloneqq \mathbb{E}\left[\int_{D} a(\boldsymbol{x}, \omega) \nabla u(\boldsymbol{x},  \omega) \cdot \nabla v(\boldsymbol{x},  \omega) \, \mathrm{d} \boldsymbol{x}\right], \label{eq4} \\
\ell(v) & \coloneqq \mathbb{E}\left[\int_{D} f(\boldsymbol{x},  \omega) v(\boldsymbol{x}, \omega)\, \mathrm{d} \boldsymbol{x}\right], \\
\mathcal{K} & \coloneqq \{ v \in V \left| \right. v(\bx,\omega) \geq g(\bx, \omega)\ { {\rm a.e.\ in\ } D \times \Omega} \}. \notag
\end{align}

\begin{theorem}\label{thm1}
Under Assumption~\ref{2.2}, for any \(f \in L^{2}\left(\Omega; L^{2}(D)\right)\), the problem~\eqref{eq5} admits a unique solution.
 \end{theorem}
\begin{proof}
To prove the existence and uniqueness of the solution to~\eqref{eq5}, we show that the bilinear form is bounded and $V$-elliptic.
Applying the Cauchy-Schwarz inequality, we get
\begin{align*}
b(u, v) &= \int_{\Omega}\int_D a\,\nabla u \cdot \nabla v \, \mathrm{d}\bx \, \mathrm{d} \mathbb{P} (\omega) \\
				& \leq a_{\max} \int_{\Omega}\int_D |\nabla u \cdot \nabla v|\, \mathrm{d}\bx \,\mathrm{d}\mathbb{P} (\omega)\\
				&\leq a_{\max}\lVert \nabla u \rVert _ {L^2(\Omega; L^2(D))}  \lVert \nabla v  \rVert _ {L^2(\Omega; L^2(D))} \\
		 &= a_{\max}  \lVert  u \rVert _ V  \lVert  v  \rVert _ V.
	\end{align*}
By Assumption~\ref{2.2}, we have 
\begin{align*}
	b(v, v) &= \int_{\Omega}\int_D a\,\nabla v \cdot \nabla v \,\mathrm{d}\bx \,\mathrm{d} \mathbb{P}(\omega) \\
			& \geq a_{\min} \int_{\Omega}\int_D \lvert \nabla v \rvert ^2 \, \mathrm{d}\bx \, \mathrm{d}\mathbb{P} (\omega) = a_{\min} \lVert  v \rVert _{L^2(\Omega; H_0^1(D))}^2. 
\end{align*}

It is readily verified that \(f \in L^{2}\left(\Omega; L^{2}(D)\right)\) implies the continuity of the linear form $\ell$
on $V$.  By applying Lemma~\ref{lem1.1}, we conclude that~\eqref{eq5} admits a unique solution.
\end{proof}

Since the diffusion coefficient \(a\), the source term \(f\) and obstacle function \(g\) are typically not known exactly 
in applications, it is useful to consider approximations for \(a\), \(f\), and \(g\). 
We now examine a perturbation of the weak formulation (\ref{eq5}) 
and analyze the magnitude of the resulting perturbation in the solution.
Replacing  $a$, $f$, and $g$ in (\ref{eq5})  by their approximations $\widetilde{a}, \widetilde{f}, \widetilde{g}\colon D \times \Omega \to  \mathbb{R}$  leads to a perturbed stochastic variational inequality: 
\begin{align}
	\label{eq8}
\widetilde{u} \in \widetilde{\mathcal{K}},\quad	\widetilde b(\widetilde{u}, v-\widetilde{u}) \geq  \widetilde \ell(v-\widetilde{u}), \quad \forall v \in \widetilde{\mathcal{K}},
\end{align}
where  $\widetilde{b}\colon V \times V \to \mathbb{R}$  and
  $\widetilde{\ell}\colon V \to \mathbb{R}$  are defined by
\[
	\begin{aligned}
	\widetilde{b}(u, v) &  \coloneqq \mathbb{E}\left[\int_{D} \widetilde{a}(\boldsymbol{x}, \omega) \nabla u(\boldsymbol{x}, \omega) \cdot \nabla v(\boldsymbol{x}, \omega)\,\mathrm{d} \boldsymbol{x}\right], \\
		\widetilde{\ell}(v) &  \coloneqq \mathbb{E}\left[\int_{D} \widetilde{f}(\boldsymbol{x}, \omega) v(\boldsymbol{x}, \omega)\,\mathrm{d} \boldsymbol{x}\right],\\
		\widetilde{\mathcal{K}} & \coloneqq \{ v \in V \left| \right. v(\bx,\omega) \geq {\widetilde{g}(\bx, \omega)} \ { {\rm a.e.\ in\ } D \times \Omega} \}.\notag
		\end{aligned}
\]

As demonstrated in~\cite{Babuska2004}, we employ a general framework to establish the well-posedness of problem (\ref{eq8}).

{
\begin{proposition}\label{prop2.2}
Let $H$ be a Hilbert space. Let $\mathcal{K},\widetilde{\mathcal{K}}\subset H$ be non-empty, closed, convex sets. 
Consider symmetric bilinear forms $\mathcal{B}, \widetilde{\mathcal{B}}\colon H \times H \to \mathbb{R}$ that are $H$-coercive and bounded: there exist { $0<\alpha_{\min} \le \alpha_{\max}$} such that
\[
	\alpha_{\min}\|v\|_H^2 \le \min\{\mathcal{B}(v,v),\widetilde{\mathcal{B}}(v,v)\}, 
\qquad
\max\{|\mathcal{B}(v,w)|,|\widetilde{\mathcal{B}}(v,w)|\}\le \alpha_{\max}\|v\|_H\|w\|_H
\]
for all $v, w \in H$. Assume moreover that $\mathcal{B}$ and $\widetilde{\mathcal{B}}$ are comparable: there exists $\gamma\ge0$ such that
\begin{equation}\label{eq2.8}
\big|\mathcal{B}(v,w)-\widetilde{\mathcal{B}}(v,w)\big|\le \gamma\,\|v\|_H\|w\|_H
\qquad\forall\,v, w \in H.
\end{equation}
Given $\mathcal{L},\widetilde{\mathcal{L}} \in H'$, let $u \in \mathcal{K},\widetilde{u} \in \widetilde{\mathcal{K}}$ solve
\begin{equation}\label{eq2.9}
\mathcal{B}(u,v-u)\ge \mathcal{L}(v-u)\quad\forall\,v \in \mathcal{K},
\end{equation}
\begin{equation}\label{eq2.9t}
\widetilde{\mathcal{B}}(\widetilde{u},v-\widetilde{u})\ge \widetilde{\mathcal{L}}(v-\widetilde{u})\quad\forall\,v \in \widetilde{\mathcal{K}}.
\end{equation}
Denote the  distances
\[
\delta_{\mathcal{K}} \coloneqq \inf_{v \in \mathcal{K}}\|\,\widetilde{u}-v\,\|_H,\qquad
\delta_{\widetilde{\mathcal{K}}} \coloneqq \inf_{v \in \widetilde{\mathcal{K}}}\|\,u-v\,\|_H,
\]
and set $e \coloneqq u-\widetilde{u}$. Then
{\small \begin{equation}\label{eq2.11}
\|e\|_H \leq  
\frac{1}{\alpha_{\min}} \left(\|\widetilde{\mathcal{L}}-\mathcal{L}\|_{H'}+\gamma\|u\|_H\right) +
\sqrt{\frac{2}{\alpha_{\min}}\!\left(\big(\alpha_{\max}\|u\|_H+\|\mathcal{L}\|_{H'}\big)\delta_{\mathcal{K}}
+\big(\alpha_{\max}\|\widetilde{u}\|_H+\|\widetilde{\mathcal{L}}\|_{H'}\big)\delta_{\widetilde{\mathcal{K}}}\right)}.
\end{equation}}
In particular, if $\mathcal{K} = \widetilde{\mathcal{K}}$, then $\delta_{\mathcal{K}} = \delta_{\widetilde{\mathcal{K}}} = 0$ and
\begin{equation}\label{eq:nogap}
\|u-\widetilde{u}\|_H\leq \frac{1}{\alpha_{\min}}\Big(\|\widetilde{\mathcal{L}}-\mathcal{L}\|_{H'}+\gamma\|u\|_H\Big).
\end{equation}
Moreover, if $\mathcal{K}$ and $\widetilde{\mathcal{K}}$ are convex cones with vertex at $0$, then
\begin{equation*}
\|u\|_H\leq \frac{1}{\alpha_{\min}}\, \|\mathcal{L}\|_{H'},\qquad
\|\widetilde{u}\|_H\leq \frac{1}{\alpha_{\min}}\, \|\widetilde{\mathcal{L}}\|_{H'},
\end{equation*}
and consequently
{\small \begin{equation}\label{eq2.13}
\|u-\widetilde{u}\|_H
\leq    \frac{1}{\alpha_{\min}}\!\left(\|\widetilde{\mathcal{L}}-\mathcal{L}\|_{H'}+\frac{\gamma}{\alpha_{\min}}\|\mathcal{L}\|_{H'}\right) +
\sqrt{\frac{2}{\alpha_{\min}}\! \Big(1+\frac{\alpha_{\max}}{\alpha_{\min}}\Big)\big(\|\mathcal{L}\|_{H'}\,\delta_{\mathcal{K}}
+\|\widetilde{\mathcal{L}}\|_{H'}\,\delta_{\widetilde{\mathcal{K}}}\big)}.
\end{equation}}
\end{proposition}

\begin{proof}
Applying Lemma~\ref{lem1.1}, we know that \(u\) and \(\widetilde{u}\) are well defined.
Denote $e \coloneqq u-\widetilde{u}$. Choose $\widetilde{u}_{\mathcal{K}} \in \mathcal{K}$ and $u_{\widetilde{\mathcal{K}}} \in \widetilde{\mathcal{K}}$ such that
$\|\widetilde{u}_{\mathcal{K}}-\widetilde{u}\|_H=\delta_{\mathcal{K}}$ and $\|u-u_{\widetilde{\mathcal{K}}}\|_H=\delta_{\widetilde{\mathcal{K}}}$ (cf. \cite[Section 3.4]{Atkinson2009Theoretical}).
Using~\eqref{eq2.9} with $v=\widetilde{u}_{\mathcal{K}}$ and~\eqref{eq2.9t} with $v=u_{\widetilde{\mathcal{K}}}$ yields
\[
\mathcal{B}\big(u,\widetilde{u}_{\mathcal{K}}-u\big)\ge \mathcal{L}\big(\widetilde{u}_{\mathcal{K}}-u\big),\qquad
\widetilde{\mathcal{B}}\big(\widetilde{u},u_{\widetilde{\mathcal{K}}}-\widetilde{u}\big)\ge \widetilde{\mathcal{L}}\big(u_{\widetilde{\mathcal{K}}}-\widetilde{u}\big).
\]
By coercivity of $\widetilde{\mathcal{B}}$,
\[
	\alpha_{\min}\|e\|_H^2\le \widetilde{\mathcal{B}}(e,e)=\widetilde{\mathcal{B}}(u,e)-\widetilde{\mathcal{B}}(\widetilde{u},e).
\]
Insert and subtract $\mathcal{B}(u,e)$, and rewrite
$\widetilde{u}_{\mathcal{K}}-u=(\widetilde{u}_{\mathcal{K}}-\widetilde{u})-e$ and
$u_{\widetilde{\mathcal{K}}}-\widetilde{u}=e-(u-u_{\widetilde{\mathcal{K}}})$ to obtain
\[
\begin{aligned}
\alpha_{\min}\|e\|_H^2
&\le \big(\widetilde{\mathcal{B}}(u,e)-\mathcal{B}(u,e)\big)+\big(\mathcal{L}(e)-\widetilde{\mathcal{L}}(e)\big) \\
&\quad +\Big(\mathcal{B}\big(u,\widetilde{u}_{\mathcal{K}}-\widetilde{u}\big)-\mathcal{L}\big(\widetilde{u}_{\mathcal{K}}-\widetilde{u}\big)\Big)
- \Big(\widetilde{\mathcal{B}}\big(\widetilde{u},\,u-u_{\widetilde{\mathcal{K}}}\big)-\widetilde{\mathcal{L}}\big(u-u_{\widetilde{\mathcal{K}}}\big)\Big).
\end{aligned}
\]
By~\eqref{eq2.8} and boundedness,
\[
\begin{aligned}
\big|\widetilde{\mathcal{B}}(u,e)-\mathcal{B}(u,e)\big| &\leq \gamma\|u\|_H\|e\|_H,\\
\big|\mathcal{L}(e)-\widetilde{\mathcal{L}}(e)\big| &\leq \|\widetilde{\mathcal{L}}-\mathcal{L}\|_{H'}\|e\|_H,\\
\big|\mathcal{B}(u,\widetilde{u}_{\mathcal{K}}-\widetilde{u})-\mathcal{L}(\widetilde{u}_{\mathcal{K}}-\widetilde{u})\big|
&\leq \big(\alpha_{\max}\|u\|_H+\|\mathcal{L}\|_{H'}\big)\,\|\widetilde{u}_{\mathcal{K}}-\widetilde{u}\|_H
=\big(\alpha_{\max}\|u\|_H+\|\mathcal{L}\|_{H'}\big)\delta_{\mathcal{K}},\\
\big|\widetilde{\mathcal{B}}(\widetilde{u},u-u_{\widetilde{\mathcal{K}}})-\widetilde{\mathcal{L}}(u-u_{\widetilde{\mathcal{K}}})\big|
&\leq \big(\alpha_{\max}\|\widetilde{u}\|_H+\|\widetilde{\mathcal{L}}\|_{H'}\big)\,\|u-u_{\widetilde{\mathcal{K}}}\|_H
=\big(\alpha_{\max}\|\widetilde{u}\|_H+\|\widetilde{\mathcal{L}}\|_{H'}\big)\delta_{\widetilde{\mathcal{K}}}.
\end{aligned}
\]
Hence
\[
	\alpha_{\min}\|e\|_H^2
\leq
\Big(\gamma\|u\|_H+\|\widetilde{\mathcal{L}}-\mathcal{L}\|_{H'}\Big)\|e\|_H
+\big(\alpha_{\max}\|u\|_H+\|\mathcal{L}\|_{H'}\big)\delta_{\mathcal{K}}
+\big(\alpha_{\max}\|\widetilde{u}\|_H+\|\widetilde{\mathcal{L}}\|_{H'}\big)\delta_{\widetilde{\mathcal{K}}}.
\]
Applying Young's inequality to the linear term in $\|e\|_H$ with parameter $\alpha_{\min}/2$ gives
\[
\frac{\alpha_{\min}}{2}\|e\|_H^2
\leq
\frac{1}{2\alpha_{\min}}\Big(\gamma\|u\|_H+\|\widetilde{\mathcal{L}}-\mathcal{L}\|_{H'}\Big)^2
+\big(\alpha_{\max}\|u\|_H+\|\mathcal{L}\|_{H'}\big)\delta_{\mathcal{K}}
+\big(\alpha_{\max}\|\widetilde{u}\|_H+\|\widetilde{\mathcal{L}}\|_{H'}\big)\delta_{\widetilde{\mathcal{K}}},
\]
which implies \eqref{eq2.11}. If $\mathcal{K}=\widetilde{\mathcal{K}}$, the distances vanish and~\eqref{eq:nogap} follows.

If $\mathcal{K}$ and $\widetilde{\mathcal{K}}$ are cones with vertex at $0$, testing \eqref{eq2.9} by $v=0$ and $v=2u$ yields 
$\mathcal{B}(u,u)=\mathcal{L}(u)$ and thus $\|u\|_H\le \alpha_{\min}^{-1}\|\mathcal{L}\|_{H'}$; the same argument for $\widetilde{u}$. 
Substituting these into \eqref{eq2.11} yields \eqref{eq2.13}.
\end{proof}

\begin{corollary}
	The distance \(\max{(\delta_{\widetilde{\mathcal{K}}}, \delta_{\mathcal{K}})} \leq \lVert \widetilde{g} - g \rVert _ H\). 
\end{corollary}
\begin{proof}
	From definition, \(u + \max{(\widetilde{g} - u, 0)}  \geq \widetilde{g}\).
	Then,
\[	 \begin{aligned}
	\delta_{\widetilde{\mathcal{K}}}\,=\,\inf_{v \in \widetilde{\mathcal{K}}}\|\, {u}-v\,\|_H & \leq \|\,u - (u + \max{(\widetilde{g} - u, 0))} \,\|_H \\
	& \leq \|\, g - \widetilde{g} \,\|_H.
 \end{aligned}
 \]
Similarly, we have \(	\delta_{\mathcal{K}}\, \leq \|\, g - \widetilde{g} \,\|_H\).
\end{proof}
}

\begin{corollary}\label{thm2}
	Let \(H = L^2(\Omega; H_0^1(D))\). Keep the assumptions stated in Theorem \ref{thm1}.  Assume the perturbed coefficient \(\widetilde{a}\) satisfies \(0 < a_{\min} \leq \widetilde{a} \leq a_{\max} < \infty\), a.e. \ on \(D \times \Omega\) and \(f, \widetilde{f} \in L^2(\Omega; L^2(D))\). Let $ u \in \mathcal{K}$ and $\widetilde{u} \in \widetilde{\mathcal{K}}$ be the unique solutions to~\eqref{eq5} and~\eqref{eq8}, respectively. 

{
		Then
{\small	\begin{align} \label{eq2.14}
\|u-\widetilde u\|_{H}
	\leq \, & \frac{1}{a_{\min}}\,\left(\|\widetilde a-a\|_{L^{\infty}(\Omega;L^{\infty}(D))}\,\|u\|_{H}+ C_{D}\,\|\widetilde f-f\|_{L^{2}(\Omega;L^{2}(D))}\right)   + \\
	&	\sqrt{\frac{2}{a_{\min}}\!\left(\big(a_{\max}\|u\|_{H}+C_{D}\|f\|_{L^{2}(\Omega;L^{2}(D))}\big)\,\delta_{\mathcal{K}}+\big(a_{\max}\|\widetilde u\|_{H}+C_{D}\|\widetilde f\|_{L^{2}(\Omega;L^{2}(D))}\big)\,\delta_{\widetilde{\mathcal{K}}}
\right)}.  \notag
\end{align}}}
In particular, if \(g=\widetilde{g}\) then \eqref{eq2.14} reduces to
{\small \begin{equation}\label{eq2.15}
\|u-\widetilde u\|_{L^{2}(\Omega;H_{0}^{1}(D))}
\,\leq\, a_{\min}^{-1}\!\left(\|\widetilde a-a\|_{L^{\infty}(\Omega;L^{\infty}(D))}\,\|u\|_{L^{2}(\Omega;H_{0}^{1}(D))}
+ C_{D}\,\|\widetilde f-f\|_{L^{2}(\Omega;L^{2}(D))}\right).
\end{equation}}
\noindent where \(C_D > 0\) is the Poincar\'{e} constant for the domain \(D\), i.e., \( \lVert v \rVert _{L^2(D)} \leq C_D \lVert v \rVert _{H_0^1(D)}\) for any \(v \in H_0^1(D)\).
\end{corollary}
	
\begin{proof}
We apply Proposition~\ref{prop2.2}. From 
\begin{align*}
| \mathbb{E}[\int_D (a-\widetilde{a}) \nabla v \cdot \nabla w\, \mathrm{d} \bx] |
& \leq \|\widetilde{a}-a\|_{L^{\infty}\left(\Omega; L^{\infty}(D)\right)} \left( \mathbb{E}[\int_D | \nabla v |^2  \,\mathrm{d} \bx] \right)^{1/2}\left( \mathbb{E}[\int_D | \nabla w |^2 \, \mathrm{d} \bx ] \right)^{1/2} \\
& =  \|\widetilde{a}-a\|_{L^{\infty}\left(\Omega; L^{\infty}(D)\right)} \| v \|_H \|w\|_H,
	\end{align*}
we see that~\eqref{eq2.8} holds with $\gamma=\|\widetilde{a}-a\|_{L^{\infty}\left(\Omega; L^{\infty}(D)\right)}$.   Then,~\eqref{eq2.14} follows from~\eqref{eq2.11}.
\end{proof}

\subsection{Karhunen-Lo\`eve expansions}

Randomness in many systems can often be effectively approximated using a small number of uncorrelated or independent random variables. In this work, we consider the approximation of the random fields $\widetilde{a}$, $\widetilde{f}$ and $\widetilde{g}$ through truncated Karhunen-Lo\`eve expansions (\cite{Maitre2010, Lord2014introduction}). 

Denote by \(\mu_a(\bx)\) the expectation of \(a(\bx, \omega)\).  Define a function $V_{a}: D \times D \to \mathbb{R}$ by
\[ V_{a}(\bx, \bx^{\prime}) \coloneq \int_{\Omega} (a(\bx,\omega) - \mu_a(\bx)) (a(\bx^{\prime},\omega) - \mu_a(\bx^{\prime})) \, \mathrm{d} \mathbb{P}(\omega)\quad
(\bx,\bx^{\prime}) \in D \times D,\]
and then define the Carleman operator $\mathcal{V}_{a}\colon L^{2}(D) \to L^{2}(D)$ by
\[(\mathcal{V}_{a}\phi)(\bx) \coloneq \int_{D} V_{a}(\bx,\bx^{\prime}) \phi (\bx^{\prime}) \, \mathrm{d} \bx^{\prime}\quad
\forall \phi  \in L^{2}(D).\]

The Carleman operator \(\mathcal{V}_a\) is a symmetric, nonnegative definite and compact integral operator \cite[Section 1.3]{Lord2014introduction}. It possesses a countable sequence \( \{ (v_k^a, \phi_k^a) \}_{k \in \mathbb{N} _+}\) of eigenpairs with \(v_1^a \geq v_2^a \geq \cdots \geq 0 \) satisfying 
\[\sum _{k = 1}^{\infty} v_k^a = \int_D \int _{\Omega} (a(\bx, \omega) - \mu_a(\bx))^2 \,\mathrm{d}\mathbb{P}(\omega) \, \mathrm{d} \bx < \infty.\]

\begin{proposition}[\cite{Schwab2006}] \label{thm2.7}
If $ a \in L^{2}(D; L^2(\Omega)) $, then there exists a sequence of random variables $ \{{\xi_k^a}\}_{k \geqslant 1} \subset L^{2}(\Omega) $ satisfying
    \begin{equation*}
        \int_{\Omega} {\xi_i^a}(\omega) \,\mathrm{d} \mathbb{P}(\omega) = 0, \quad
        \int_{\Omega} {\xi_i^a}(\omega) {\xi_j^a}(\omega) \,\mathrm{d}\mathbb{P}(\omega) = \delta_{ij} \quad
        \forall i, j \geqslant 1
    \end{equation*}
and such that the random field $a(\cdot,\cdot)$ can be expanded in $ L^{2}(D; L^2(\Omega)) $ as
    \begin{equation} \label{eq2.17}
        a(\bx,\omega) = \mu_{a}({\bx}) + \sum_{k=1}^{\infty} \sqrt{v_{k}^a} \phi_{k}^{a}(\bx) \xi_{k}^{a}(\omega),
    \end{equation}
where  \( \{ (v_k^a, \phi_k^a)\} _{k \in \mathbb{N}_+} \)
is the eigenpair sequence of the Carleman operator \(\mathcal{V}_a\).
Moreover, for ${v_k}^{a} \neq 0 $,
    \begin{equation*}
        \xi_{k}^{a}(\omega) \coloneqq \frac{1}{\sqrt{v_k^a}} \int_{D} (a(\bx,\omega) - \mu_a) \phi_{k}^{a}(\bx) \, \mathrm{d}\bx \quad
        \forall k \geqslant 1.
    \end{equation*}
\end{proposition}

The formula \eqref{eq2.17} is called Karhunen-Lo\`eve expansion of the random field \(a \in L^2(D; L^2(\Omega))\).

Based on the expansion \eqref{eq2.17}, for any \(M_a \in \mathbb{N}\), we define the truncated diffusion coefficient by
\begin{align}\label{truncate_a}
	{a}_{M_a}({\bx}, \omega) = \mu_{a}({\bx})+\sum_{k=1}^{M_a} \sqrt{v_{k}^{a}} \phi_{k}^{a}({\bx}) \xi_{k}^a(\omega).
\end{align}
Similarly, the source term \(f\) is approximated by the truncated expansion 
\begin{equation}\label{truncate_f}
	{f}_{M_f}({\bx}, \omega) = \mu_{f}({\bx})+\sum_{k=1}^{M_f} \sqrt{v_{k}^{f}} \phi_{k}^{f}({\bx}) \xi_{k}^{f}(\omega),
\end{equation}
{ and the truncated expansion of obstacle function \(g\)
\begin{equation}\label{truncate_g}
	{g}_{M_g}({\bx}, \omega) = \mu_{g}({\bx})+\sum_{k=1}^{M_g} \sqrt{v_{k}^{g}} \phi_{k}^{g}({\bx}) \xi_{k}^{g}(\omega).
\end{equation}}

We denote the truncated approximations of the coefficient,  source term and obstacle function, by  \(\widetilde{a} = a_{M_a}\), \(\widetilde{f} = f_{M_f}\) and \(\widetilde{g} = g_{M_g}\) and use the notation \(\widetilde{a}, \widetilde{f}, \widetilde{g}\) throughout to refer to the Karhunen-Lo\`eve approximations of \(a, f, g\) with truncation levels \(M_a, M_f, M_g\), respectively. 

To apply {Corollary~\ref{thm2}}, uniform convergence of the truncated diffusion coefficient is required. However, {Proposition~\ref{thm2.7}} establishes only $L^2$-convergence of the truncated term:
\begin{equation}
\big\| a - a_{M_a} \big\|_{L^2(\Omega; L^2(D))} \to 0 \quad \text{as} \quad M_a \to \infty.
\end{equation}

To achieve the required uniform convergence for the Karhunen-Lo\`eve expansion:
\begin{equation*}
\big\| a- a_{M_a}\big\|_{L^\infty(\Omega ; L^\infty(D))} \to 0 \quad \text{as} \quad M_a \to \infty,
\end{equation*}
the analyticity of \(a(\bx,\omega)\) in its physical variables, combined with the uniform boundedness of the corresponding sequence \(\{\xi_k^a\}_{k \in \mathbb{N}_+}\)  is sufficient to ensure the uniform convergence of its Karhunen-Lo\`eve expansion.

\begin{proposition} (\cite{Todor2007})\label{thm2.8}
If the covariance \(V_a\) of \(a\) is piecewise analytic on \(D \times D\) and \( \{\xi_k^a\}_{k \in \mathbb{N}_+}\) is bounded uniformly,
then
	\begin{equation}
	\|a - a_{M_a}\|_{L^{\infty}(\Omega; L^\infty(D))} \leq c_{a} \exp\left(-c_{1,a} M_a^{1/n}\right) \quad \forall M_a \in \mathbb{N},
	\end{equation}
where \(n\) is the dimension of physical space, \(c_a\) is a bound of \(\{\xi_k^a \}_{k \in \mathbb{N}_+}\), \(c_{1,a}\) is a positive constant.
\end{proposition}
{The uniform convergence of the Karhunen-Lo\`eve expansion can be established under the conditions given in \cite{Babuska2004}. 
In particular, it requires that the random variables \( \{\xi_k^a\}_{k \in \mathbb{N}_+} \) are uniformly bounded, 
the eigenfunctions \( \{\phi_{k}^{a}\}_{k \in \mathbb{N}_+} \) are sufficiently smooth (which holds when the covariance function is smooth), 
and that the eigenpairs satisfy the decay condition  \(\sqrt{v_{k}^{a}} \, \| \phi_{k}^{a} \|_{L^{\infty}(D)} = O\!\left(\frac{1}{1+k^s}\right), \,s>1\).  Moreover, in certain special cases, the uniform convergence of the Karhunen-Lo\`eve expansion can also be guaranteed by \cite[Corollary 9.35]{Lord2014introduction}. Specifically, if the random variables \( \{\xi_k^a\}_{k \in \mathbb{N}_+} \) remain uniformly bounded and the covariance random field \(a\) follows a Whittle--Matérn covariance structure, the result holds as well.}

{In fact, if there is a deterministic sequence \(b_k \geq 0\) with \(\operatorname*{ess\,sup}_{\omega\in\Omega} |\xi_k^a(\omega)| \;\le\; b_k\) for all \(k\) and \(\sum_{k=1}^{\infty}\sqrt{v_k^a} b_k \| \phi_k^a \|_{L^{\infty}(D)} < \infty\), then by the Weierstrass M-test (\cite[Theorem 7.10]{Rudin1976}), the Karhunen-Lo\`eve expansion is uniformly absolutely convergent.}

To establish the existence of a solution to equation (\ref{eq8}), we require the uniform coercivity of the bilinear form \(\widetilde{b}\). While Assumption~\ref{2.2} ensures that \( a \) is uniformly bounded, this does not necessarily guarantee that the lower and upper bounds of \( a \) are preserved in its finite-term truncation \( a_{M_a}\) unless \(M_a\) is sufficiently large (Proposition~\ref{thm2.8}). However, if \( \{\xi_k^a\}_{k \in \mathbb{N}_+} \) are independent, then under Assumption~\ref{2.2}, these bounds hold for \( \widetilde{a} \) as well for any \(M_a \geq 0\).  As shown in~\cite[Remark 3.3]{Todor2007}, if we denote \( \mathcal{F}_{M_a} \subset \mathcal{F} \) as the \( \sigma \)-algebra generated by the random variables \( \xi_1, \ldots, \xi_{M_a}\), it follows \[\mathbb{E}[a|\mathcal{F}_{M_a}] = a_{M_a}.\] 
The monotonicity of conditional expectations ensures then that the lower and upper
bounds of \({a}\) hold for \(a_{M_a}\) too, i.e.
\[0<a_{\min} \leq {a}_{M_a}({\bx},\omega)\leq a_{\max}<\infty \quad \text {a.e.\ in } D \times \Omega.\]
Since $f$ is expanded in the form of Karhunen-Lo\`eve, and $f \in L^{2}\left(\Omega; L^{2}(D)\right)$, it follows from Proposition \ref{thm2.7} that 
\begin{equation*}
\|f - \widetilde{f}\|_{L^{2}\left(\Omega; L^{2}(D)\right)} \to 0 \text{ as } M_f \to \infty.
\end{equation*}

We also need show the 
\[	\|\, g - g_{M_g}\,\| _{L^2(\Omega, H^1(D))}  \to 0 \text{ as }  M_g \to \infty.\]
This can be ensured if \(g \in L^2(\Omega, H^1(D))\) (\cite[Theorem 7.2.8]{Hsing2015}, \cite{Schwab2006}).

Consequently, problem~\eqref{eq8} is well-posed.  For simplicity, we assume that  $\xi_{k}^a$,  $\xi_{k}^f$ and $\xi_{k}^g$   have the same distribution and write
$\xi_k = \xi_{k}^a=\xi_{k}^f=\xi_{k}^g$ (\cite[page 389]{Lord2014introduction}) and denote $M = \max \{M_a, M_f, M_g\}$.

\subsection{The case of a finite-dimensional noise} 

{The weak form of the problem~(\ref{eq5}) is not an ideal starting point for numerical approximation, as it involves integration over the abstract probability space \( \Omega \) with respect to the measure \( \mathbb{P} \). To enable efficient numerical discretization, it is advantageous to express the randomness using a finite number of independent random variables. To this end, it is customary to transform the problem onto a parametric domain \( \Gamma \subset \mathbb{R}^M \), where the randomness is encoded by a finite set of independent random variables \( \boldsymbol{\xi} = [\xi_1, \ldots, \xi_M]^\top \).
}

Specifically, we focus on the transformation of the bilinear form
\[    b(u, v) = \int_\Omega \int_D a(\bx, \omega) \nabla u(\bx, \omega) \cdot \nabla v(\bx, \omega) \, \mathrm{d}\bx \, \mathrm{d}\mathbb{P}(\omega).\]
A similar discussion can be made for \( \ell(v) \). 

{To make this formulation more tractable, we assume that the coefficient \( \widetilde{a} \), source term \( \widetilde{f} \) and obstacle function \(\widetilde{g}\)  are functions of finitely many random variables \( \xi_k \colon \Omega \to \Gamma_k \subset \mathbb{R} \), for \( 1 \leq k \leq M \), as is the case for truncated Karhunen-Lo\`eve expansions (see Section~2.4). In this setting, each realization \( \omega \in \Omega \) corresponds to a point \( \by = [y_1, \dots, y_M]^\top \in \Gamma \coloneqq \Gamma_1 \times \cdots \times \Gamma_M \subset \mathbb{R}^M \), and the stochastic functions \( \widetilde{a}(\bx, \omega) \), \(\widetilde{f}(\bx, \omega) \), \(\widetilde{g}(\bx, \omega) \) and \( \widetilde{u}(\bx, \omega) \) can be equivalently represented as deterministic parametric functions \( \widetilde{a}(\bx, \by) \), \( \widetilde{f}(\bx, \by) \),   \( \widetilde{g}(\bx, \by) \) and \( \widetilde{u}(\bx, \by) \). This setting is referred to as the finite-dimensional noise case~(\cite[page 394]{Lord2014introduction}).}

\begin{definition}[Finite-dimensional noise] 
Let  $\xi_k\colon \Omega \to \Gamma_{k}$, $ k=1,\ldots,M$, be real-valued random variables, $M<\infty$.
A function  $v \in L^{2}\left(\Omega; L^{2}(D)\right)$  of the form  $v(\boldsymbol{x}, \boldsymbol{\xi}(\omega))$  for
	   $\boldsymbol{x} \in D$ and $\omega \in \Omega $, where $\boldsymbol{\xi}=\left[\xi_{1}, \ldots, \xi_{M}\right]^{\top}\colon \Omega \to \Gamma \subset \mathbb{R}^{M}$  and 
$\Gamma \coloneqq \Gamma_{1} \times \cdots \times \Gamma_{M}$, is called a finite-dimensional or an $M$-dimensional noise.
\end{definition}

In what follows we use $p(\boldsymbol{y})$ to denote the joint probability density of $\boldsymbol{\xi} = \left[\xi_{1}, \ldots, \xi_{M}\right]^{\top}$, and \(\Gamma = \Pi _{k=1}^{M}\Gamma_{k}\). Each $\xi_i$ is assumed to have a density function $p_i(y_i)\colon \Gamma _i \to \mathbb{R}^{+}$ for $i = 1,\ldots, M$.
If $\{\xi_i\}_{i=1}^{M}$ are independent, then \(p(\boldsymbol{y}) = \Pi _ {i=1}^{M} p_i(y_i)\).

For a finite-dimensional noise $v = v(\boldsymbol{x}, \boldsymbol{\xi})$, we can compute its expectation using a change of variables. If $p$  is the joint density of $\boldsymbol{\xi}$, then
\[
	\|v\|_{L^{2}\left(\Omega; L^{2}(D)\right)}^{2}=\mathbb{E}\left[\|v\|_{L^{2}(D)}^{2}\right] = \int_{\Gamma} p(\boldsymbol{y})\|v(\cdot, \boldsymbol{y})\|_{L^{2}(D)}^{2} \, \mathrm{d}\boldsymbol{y},
	\]
where $y_{k}\coloneq  \xi_{k}(\omega)$ and $\boldsymbol{y}=\left[y_{1}, \ldots, y_{M}\right]^{\top}$. 
Thus, any random field
$v(\boldsymbol{x}, \boldsymbol{\xi}) \in L^{2}\left(\Omega; L^{2}(D)\right)$ can be associated with a function
 $v=v(\boldsymbol{x}, \boldsymbol{y})$  belonging to the weighted  $L^{2}$  space
\[
	L_{p}^{2}\left(\Gamma; L^{2}(D)\right) \coloneqq \left\{v\colon D \times \Gamma \to \mathbb{R} \mid \int_{\Gamma} p(\boldsymbol{y})\|v(\cdot, \boldsymbol{y})\|_{L^{2}(D)}^{2} \, \mathrm{d} \boldsymbol{y}<\infty\right\} .
	\]  
If the diffusion coefficient \(a(\boldsymbol{x}, \omega)\), the source \(f(\boldsymbol{x}, \omega)\) and obstacle function \(g(\bx, \omega)\)  are finite-dimensional noises, we can combine \(\boldsymbol{\xi}_a\), \(\boldsymbol{\xi}_f\) and  \(\boldsymbol{\xi}_g\) to form a finite-dimensional random variable \(\boldsymbol{\xi} = [\boldsymbol{\xi}_a, \boldsymbol{\xi}_f, \boldsymbol{\xi}_g]\).
In this case, $a(\boldsymbol{x}, \cdot)$, $f(\boldsymbol{x}, \cdot)$ and  $g(\boldsymbol{x}, \cdot)$ are $\sigma(\boldsymbol{\xi})$-measurable, and the solution $u(\boldsymbol{x}, \cdot)$ of (\ref{eq5}) is also $\sigma(\boldsymbol{\xi})$-measurable (\cite[Theorem 3.3]{Kornhuber2014}, \cite[Proposition 1.1]{Gwinner2000}).
Under the assumption of finite-dimensional noise, $u(\boldsymbol{x},\omega)$ can be described by the same random variable 
$\boldsymbol{\xi}$ by Doob-Dynkin lemma (cf.\ \cite[Proposition 3]{Rao2006}). Thus, we denote the solution as \(u(\boldsymbol{x},\boldsymbol{y})\) to explicitly represent \(u(\boldsymbol{x},\omega)\).

Returning to the weak form (\ref{eq8}), we obtain its deterministic equivalent.
 Let  $\widetilde{a}$, $\widetilde{f}$ and $\widetilde{g}$ be the truncated Karhunen-Lo\`eve expansions 
 (\ref{truncate_a}), (\ref{truncate_f}) and \eqref{truncate_g}, define
\[W \coloneqq L_{p}^{2}\left(\Gamma; H_0^{1}(D)\right) = \left\{v \colon D \times \Gamma \to \mathbb{R} \mid \int_{\Gamma} p(\boldsymbol{y})\|v(\bx, \boldsymbol{y})\|_{H_0^{1}(D)}^{2} \, \mathrm{d} \boldsymbol{y}<\infty \right\},
\]
\[\widetilde{\mathcal{K}} \coloneqq \{v \in W \mid v(\boldsymbol x,\boldsymbol y) \geq \widetilde{g}(\boldsymbol x, \by)\text{ a.e.\,in }D \times \Gamma \}.
\]
The variational inequality (\ref{eq8})  on  $D \times \Gamma$ is
\begin{equation}
	\label{eq11}
 \widetilde{u} \in \widetilde{\mathcal{K}} , \quad \widetilde b( \widetilde u, v-\widetilde u) \geq  \widetilde \ell(v- \widetilde u), \quad \forall \, v \in \widetilde{\mathcal{K}} ,
\end{equation}
where  $\widetilde{b}\colon W \times W \to \mathbb{R}$  and  $\widetilde{\ell}\colon W \to \mathbb{R}$  are defined by
\begin{align}
	\label{eq12}
\widetilde{b}(u, v) & \coloneqq \int_{\Gamma} p(\boldsymbol{y}) \int_{D} \widetilde{a}(\boldsymbol{x}, \boldsymbol{y}) \nabla u(\boldsymbol{x}, \boldsymbol{y}) \cdot \nabla v(\boldsymbol{x}, \boldsymbol{y}) \,\mathrm{d} \boldsymbol{x} \, \mathrm{d} \boldsymbol{y}, \\
\widetilde{\ell}(v) & \coloneqq \int_{\Gamma} p(\boldsymbol{y}) \int_{D} \widetilde{f}(\boldsymbol{x}, \boldsymbol{y}) v(\boldsymbol{x}, \boldsymbol{y})\, \mathrm{d} \boldsymbol{x} \, \mathrm{d} \boldsymbol{y} .\notag
\end{align}
The weight $ p\colon \Gamma \to \mathbb{R}^{+}$ is the joint density of  $\boldsymbol{\xi}=\left[\xi_{1}, \ldots, \xi_{M}\right]^{\top}$, 
where  $M\coloneqq\max \{M_a, M_f, M_g\}$  and $ \xi\colon \Omega \to \Gamma \subset \mathbb{R}^{M} $. Under these assumptions, $\widetilde{u} \in \widetilde{\mathcal{K}}$ is uniquely determined by  (\ref{eq11}). Consequently, we have transformed the stochastic problem (\ref{eq8}) into the deterministic problem (\ref{eq11}).

Transforming a stochastic variational inequality into its deterministic parametric form is a crucial step that enables the application of finite element methods to approximate solutions.
Working over $\Gamma$ instead of $\Omega $, we have
\[
	\widetilde{a}(\boldsymbol{x}, \boldsymbol{y})=\mu_{a}(\boldsymbol{x})+\sum_{k=1}^{M_a} \sqrt{v_{k}^{a}} \phi_{k}^{a}(\boldsymbol{x}) y_{k}, \quad \boldsymbol{x} \in D, \boldsymbol{y} \in \Gamma
\]
and a similar expression for $\widetilde{f}$.
We can ensure the existence of a unique solution of  (\ref{eq11}) under the following conditions.
 
\begin{assumption}\label{assumption5.4}
There exist two positive constants $ \widetilde{a}_{\min }\le\widetilde{a}_{\max }$ such that
\[
	0<\widetilde{a}_{\min } \leq \widetilde{a}(\boldsymbol{x}, \boldsymbol{y}) \leq \widetilde{a}_{\max }<\infty, \quad \text { a.e.\,in } D \times \Gamma.
\]
\end{assumption}

Recall that if $\xi_1,\ldots, \xi_{M_a}$ are independent, then Assumption~\ref{2.2} implies
\[ {a}_{\min} \leq  \widetilde{a}(\boldsymbol{x}, \boldsymbol{y}) \leq a_{\max}.\]

\begin{theorem}\label{9.43}
Under Assumption~\ref{assumption5.4}, for any $\widetilde{f} \in L_{p}^{2}\left(\Gamma; L^{2}(D)\right)$, the variational inequality \eqref{eq11} admits a unique solution.
\end{theorem}

Since our primary focus is on developing numerical methods for the stochastic obstacle problem rather than modeling the problem itself, we consider synthetic expansions of the form~\eqref{truncate_a}. In these expansions, the pairs \((\sqrt{v_k^a}, \phi_k^a)\) are carefully selected so that the solution representation has specific, desirable properties. {Notably, in our numerical experiments, the coefficient expansions are synthetically constructed and are not derived from predefined covariance operators. Hence, they do not correspond to classical Karhunen-Lo\`eve expansions in the strict sense.}

\section{Finite Element Spaces}\label{sec3}

In this section, we begin by examining standard finite element spaces over the spatial domain
{ \( D \subset \mathbb{R}^n \)} and the outcome set \( \Gamma \subset \mathbb{R}^M \) independently. 
Next, we introduce tensor product finite element spaces on \(D \times \Gamma\), 
which will serve as the foundation for approximating solutions to problem (\ref{eq11}). 
Finally, we conclude with a discussion of key approximation properties of the tensor product finite element spaces.

\subsection{Finite element spaces on the spatial domain}
\setcounter{equation}0

We begin by considering the finite element spaces defined over the spatial domain \(D\). If $D$ is a polyhedral domain, then it can be partitioned into triangles and quadrilaterals. For a general domain, the partition may be more complicated. In this paper, we assume the domain $D$ is a polygon, and consider to partition $D$ into triangles $\tau$. Let \(\{\mathcal{T}_h ^D\}_{h>0}\) represent a regular family of triangulations of \(D\), such that \(\bar{D} = \cup _{\tau \in \mathcal{T}_{h}^D} \bar{\tau}\). Define \(h = \max_{\tau \in \mathcal{T} _h ^D} h_\tau\), where \(h_\tau\) denotes the diameter of the element \(\tau\). 
We assume the triangulations of $D$ are regular, which means there exists a positive constant \(\sigma\) such that \(h_\tau / \rho_\tau \leq \sigma\) for all \(\tau \in \mathcal{T}_h ^D\), where \(\rho_\tau\) is the diameter of the largest inscribed circle in \(\tau\). Moreover, the mesh parameter \(h\) approaches zero (cf.\ \cite[Section 4.4]{Brenner1994}, \cite[page 124]{Ci1978}, \cite[page 411]{Atkinson2009Theoretical}).

Consider the finite-dimensional space \(X^h \subset H_0^1(D)\), consisting of piecewise linear continuous functions defined on \(\mathcal{T}_h ^D\).
We define interpolation operator  \(\Pi_h \) by 
\[\Pi_h \colon H^2(D) \cap H_0^1(D) \to X^h, \qquad \Pi_h v = \sum _{i = 1}^{I} v(\bx _i)\varphi_i, \]
where  $\left\{\bx_{i}\right\}_{i=1}^{I} \subset \bar{D}$ is the set of interior node points in \(D\) and  $\left\{\varphi_{i}\right\}_{i=1}^{I}$ is the corresponding finite element basis functions of the corresponding nodes. Then (\cite[Theorem 10.3.9]{Atkinson2009Theoretical}, \cite{Babuska2004})  the following approximation properties hold:

There exists a constant $C>0$ independent of \(h\) such that for any \(v\in H^2(D) \cap H_0^1(D)\),
\begin{equation}
	\label{interpolation}
	  \|v - \Pi_h v \|_{m,D} \leq C h^{2-m} |v|_{2,D}, \quad m = 0,1.
\end{equation}

\subsection{Tensor product spaces on \(\Gamma\subset \mathbb{R}^M\)}

Next, we consider the finite-dimensional space defined on \(\Gamma \subset \mathbb{R}^M\) (\cite{Babuska2004,Shen2019}). Let $\Gamma = \Pi _{k= 1}^M \Gamma_k$ be as in subsection 2.5. Consider a partition of \(\Gamma\) consisting of a finite number of boxes, $B_{\gamma}= \Pi_{k=1}^{M}[a_k^{\gamma},b_k^{\gamma}] \subset \mathbb{R}^M$, $[a_k^{\gamma},b_k^{\gamma}]\subset \overline{\Gamma_k}$ for \(k = 1, \ldots,M\), such that, for some finite index set \(I\),
\[
\overline{\Gamma} = \bigcup_{\gamma \in I} B_{\gamma} = \bigcup_{\gamma \in I} \prod_{k=1}^{M} [a_k^\gamma, b_k^\gamma],
\]
where \(\mathring{B_{\gamma _ 1}} \cap \mathring{B_{\gamma _2}} = \varnothing\) for distinct \(\gamma _1, \gamma _2 \in I\).
The maximum grid size parameter \(0 < s_k < 1\) is defined as
\[ s_k = \max \left\{b_k^\gamma - a_k^\gamma \,\,|\,\, \gamma \in I \right\}, \]
for \(k = 1,\ldots, M, \) and we set \(\tilde{s} = (s_1,\ldots,s_M)\).
For each nonnegative integer multi-index \(q = \left(q_1, \ldots, q_M\right)\), consider the finite element approximation space of  piecewise polynomials of a degree at most \(q_k\) in each direction \(y_k\), denoted by \(Y_{\tilde{s}}^q \subset L^2(\Gamma)\). If \(\varphi \in Y_{\tilde{s}}^q\), then its restriction to each of the partition boxes satisfies
\[
\left.\varphi\right|_{B_{\gamma}} \in \operatorname{span}\left\{\prod_{k=1}^{M} y_k^{\alpha_k} \,\,|\,\, \alpha_k \in \mathbb{N}, \alpha_k \leq q_k, \, k = 1, \ldots, M\right\}.
\]

{The finite element spaces \(Y_{\tilde{s}}^q\) satisfy (cf.\ \cite[Section 4.6]{Brenner1994}, \cite{Babuska2004}) the following approximation property: let \(\Pi_s\) be the standard tensor product interpolation operator.  Then, there are positive constants \(\{C_k\}_{k=1}^{M}\) depending only on \(q_k\) such that for any \(v \in H^{q+1}(\Gamma)\),
\begin{equation*}
\|v-\Pi_s v\|_{L^2(\Gamma)}\le\sum_{k=1}^{M} C_k {s_k}^{q_k+1}\|\partial_{y_k}^{q_k+1} v\|_{L^2(\Gamma)}.
\end{equation*}
}

Similarly for \(\mathcal{T}_h ^D\), we will use the \(\mathcal{T}_{\tilde{s}} ^\Gamma\) to denote the division of \(\Gamma\) with respect to \(\tilde{s}\).

Additionally, if we define \(s = \max \{| b_k^\gamma - a_k^\gamma| \colon \gamma \in I, k = 1, \dots, M\}\),  it follows that
\begin{equation}
	\label{Tensor_interpolation}
	\|v - \Pi_s v\|_{L^2(\Gamma)} \leq s^{\lambda} \sum_{k=1}^{M} C_k \|\partial_{y_k}^{q_k+1} v\|_{L^2(\Gamma)},
\end{equation}
where \(\lambda = \min_{1 \leq j \leq M} \{q_j + 1\}\).

In this paper, we select piecewise polynomials of degree 1 in each direction \(y_n\), i.e., \(q = (1, \dots, 1)\), and for simplicity, we denote \(Y^s = Y_{\tilde{s}}^q\).

Having defined the finite element spaces on \(D\) and \(\Gamma\), we now construct the tensor product space over the product domain \(D \times \Gamma.\)

\subsection{Tensor product finite element spaces on \(D \times \Gamma\)}

Combining the finite element spaces \(X^h\) and \(Y^s\), we construct the following tensor product space:
\[
	 X^h \otimes Y^s  =   \operatorname{span}  \left\{ \varphi _i (\bx) \psi _j(\by) \mid \varphi _i(\bx) \in X^h,  \psi _j(\by) \in Y^s \right\}.
\]

{
Let $\{\boldsymbol{x}_i\}_{i=1}^I \subset \bar{D}$ be the set of interior spatial nodes from the triangulation $\mathcal{T}_h^D$ and $\{\boldsymbol{y}_j\}_{j=1}^J \subset \Gamma$ be the set of parameter nodes from the partition $\mathcal{T}_s^\Gamma$, which are the vertices of the boxes $B_\gamma$. 
The finite element nodes of $X^h \otimes Y^s$ are obtained as the Cartesian product of the spatial and parameter nodes:
\[
\{(\boldsymbol{x}_i, \boldsymbol{y}_j) \in D \times \Gamma \mid i = 1,\ldots,I, \ j = 1,\ldots,J\}.
\]
The corresponding basis functions of $X^h\otimes Y^s$ are given by the tensor products
\[
\Phi_{ij}(\boldsymbol{x}, \boldsymbol{y}) = \varphi_i(\boldsymbol{x}) \psi_j(\boldsymbol{y}), \quad i = 1,\ldots,I, \ j = 1,\ldots,J,
\]
where $\varphi_i$ are the piecewise linear finite element basis functions of $X^h$ and $\psi_j$ are the piecewise multilinear finite element basis functions of $Y^s$. This tensor product framework generates a total of $I \times J$ basis functions, facilitating the approximation of solutions over the product domain $D \times \Gamma$.}
\section{Stochastic Finite Element Methods}\label{sec4}
\setcounter{equation}0

In this section, we consider a fully discrete method. In the fully discrete method, we discretize both the spatial and parameter domains.

\subsection{Fully discrete method}

Let $W^{hs} = X^h \otimes Y^s$ and define
\[ \widetilde{\mathcal{K}}^{hs} =  \left\{ { v^{hs} \in W^{hs}} \mid \right.\\
   \left.   v^{hs}(\bx,\by)  \geq \widetilde{g}(\bx, \by)  \text{ for any node } (\bx, \by)   \right\}.\]

Then, the fully discrete method to solve (\ref{eq11}) is
\begin{equation}  \label{eq4.1}
	\widetilde{u}^{hs} \in \widetilde{\mathcal{K}}^{hs}, \quad	\widetilde b( \widetilde u^{hs}, v^{hs} - \widetilde u^{hs}) \geq  \widetilde \ell(v^{hs} - \widetilde u^{hs}), \quad \forall { v^{hs} \in \widetilde{\mathcal{K}}^{hs}},
	\end{equation}
where  $\widetilde{b}\colon W \times W \to \mathbb{R}$  and  $\widetilde{\ell}\colon W \to \mathbb{R}$  are defined as in (\ref{eq12}).

{ Clearly, $\widetilde{\mathcal{K}}^{hs}$ is a closed convex subset of $W^{hs}$.
}

\begin{theorem}\label{thm4.1}
 If the conditions of Theorem~\ref{9.43} hold, then (\ref{eq4.1}) admits a unique solution.
\end{theorem}

Now we provide an error analysis for the fully discrete method.  For this purpose, we recall
the following generalized Céa lemma (\cite[page 445]{Atkinson2009Theoretical}).

\begin{lemma}[\cite{Atkinson2009Theoretical}]	\label{lem6.3}
Let $ V$  be a real Hilbert space and  $K \subset V$  be non-empty, closed and convex.
Let  $V^{h} \subset V $ be a finite element space and  $K^{h} \subset V^{h}$ 
be non-empty, closed and convex. Assume the bilinear form  $b(\cdot, \cdot)$  
is bounded and  $V$-elliptic, and $l \in V^{\prime}$.

If  $u$  is a solution of
\[	u \in K, \quad b(u, v-u) \geq(f, v-u) \quad \forall v \in K, \]
and  $u^{h} \in K^{h}$  is the finite element solution defined by
\[u^h\in K^h, \quad b\left(u^h,v^h-u^h\right)\ge\left(f,v^h-u^h\right) \quad \forall v^{h} \in K^h,\]
then
\begin{equation}
\left\|u-u^{h}\right\|^{2} \leq c \inf _{v \in K} R\left(v-u^{h}\right)+c \inf _{v^{h} \in K^{h}}\left[R\left(v^{h}-u\right)+c_{1}\left\|u-v^{h}\right\|^{2}\right], \label{4.1a}
\end{equation}
where
\[
	R(v)=b(u, v)-\ell (v).
\]

\end{lemma} 

Observe that if $K^h \subset K$, then the inequality \eqref{4.1a} is reduced to 
\[\lVert u - u^h \rVert ^2 \leq c \inf _{v^{h} \in K^{h}}\left[R\left(v^{h}-u\right)+c_{1}\left\|u-v^{h}\right\|^{2}\right].  \] 

{
\begin{theorem}\label{thm4.3}
Let Assumption \ref{assumption5.4} hold, and let $\widetilde{u}$ and $\widetilde{u}^{hs}$ be solutions of the problem (\ref{eq11}) and (\ref{eq4.1}). Assume  $\widetilde{f} \in L_{p}^{2}\left(\Gamma; L^{2}(D)\right), \widetilde{u} \in H_p^2(\Gamma; H^2(D))$ and $ \widetilde{g} \in {H_p^{2}\left(\Gamma; H^{2}(D)\right)}$. Then 
\[	
	\left\|\widetilde{u}-\widetilde{u}^{hs}\right\|_{W} \leq  c (h+s) ,	
\]
where  $c$ depending only on $ \lVert\widetilde{u}\rVert_{H _p^{2}\left(\Gamma; H^{2}(D)\right)}, \lVert\widetilde{f}\rVert_{L_p^2(\Gamma;L^{2}(D))}$ and $ \lVert\widetilde{g}\rVert_{H _p^{2}\left(\Gamma; H^{2}(D)\right)}$.
\end{theorem}
\begin{proof}
Recall that $W = L_{p}^{2}\left(\Gamma; H_0^{1}(D)\right)$. Assumption \ref{assumption5.4} implies
\[\widetilde{a}_{\min }\left\|\widetilde{u}-\widetilde{u}^{hs}\right\|_{W}^2 \leq \widetilde{b}\left(\widetilde{u}-\widetilde{u}^{hs}, \widetilde{u}-\widetilde{u}^{hs}\right) \leq \widetilde{a}_{\max }\left\|\widetilde{u}-\widetilde{u}^{hs}\right\|_{W}^2 .
\]
Denote
\[
	\widetilde{R}(v)=\widetilde{b}\left(\widetilde{u}, v\right)- \widetilde{\ell} \left(v\right).
\]
First, we bound the term $\widetilde{R}(v)$:
\begin{align*}
\widetilde{R}(v) &= \widetilde{b}\left(\widetilde{u}, v\right)- \widetilde{\ell}\left(v\right) \\
		&= \int_{\Gamma} p(\by) \int_{D} \left[ \widetilde{a}(\bx,\by) \,\nabla \widetilde{u} \cdot \nabla v -  \widetilde{f}(v) \right]\,\mathrm{d} \bx\, \mathrm{d} \by.
		\end{align*}
Applying integration by parts, we have 
\begin{align*}
\widetilde{R}(v)& =  \int_\Gamma p(\by) \int_{D}\left[-\operatorname{div}(\widetilde{a}(\bx,\by) \, \nabla \widetilde{u})-\widetilde{f}\right] v\,\mathrm{d} \bx\,  \mathrm{d}  \by \\
& \leq \left\|-\operatorname{div}(\widetilde{a}\,\nabla \widetilde{u})-\widetilde{f}\right\|_{L_p^2(\Gamma;L^{2}(D))}\|v\|_{L_p^2(\Gamma;L^{2}(D))}.
\end{align*}
Lemma~\ref{lem6.3} implies
\begin{align*}
\left\|\widetilde{u}-\widetilde{u}^{hs}\right\|_{W}^{2} & \leq c \inf _{v^{hs} \in \mathcal{K}^{hs}}\left[\widetilde{R}(v^{hs}-\widetilde{u})+\left\|\widetilde{u}-v^{hs}\right\|_{W}^{2}\right] + c \inf _{v \in \mathcal{K}} \widetilde{R}(v-\widetilde{u}^{hs}) \\
&  \leq c\left\{  \inf _{v^{hs} \in \mathcal{K}^{hs}}  \left[\left\|    \widetilde{u}-v^{hs}\right\|_{W}^{2} +   \left\|-\operatorname{div}(\widetilde{a}\, \nabla \widetilde{u})-\widetilde{f}\right\|_{L_p^2(\Gamma;L^{2}(D))}\left\|v^{hs}-\widetilde{u}\right\|_{L_p^2(\Gamma;L^{2}(D))} \right] \right. \\
&\quad{} \left. + \left\|-\operatorname{div}(\widetilde{a}\, \nabla \widetilde{u})-\widetilde{f}\right\|_{L_p^2(\Gamma;L^{2}(D))} \inf _{v \in \mathcal{K}} \left\|v-\widetilde{u}^{hs} \right\|_{L_p^2(\Gamma;L^{2}(D))}\right\}.
\end{align*}
Apparently, \( \Pi_s \Pi_h  \widetilde{u} \in \widetilde{\mathcal{K}}^{hs}\), where \(\Pi_h\) is the continuous piecewise linear interpolation operator \eqref{interpolation}, and \(\Pi_s\) is the tensor product interpolation operator \eqref{Tensor_interpolation}.

According to finite element interpolation error estimates~\eqref{interpolation}, we have 
\[ \|v - \Pi_h v \|_{m,D} \leq C h^{2-m} |v|_{2,D}, \quad m = 0,1.\]
It follows that
\[ \|v - \Pi_h v \|_{L_p^2(\Gamma;H^{m}(D))} \leq C h^{2-m} |v|_{L_p^2(\Gamma;H^2(D))}, \quad m = 0,1.\]
From the triangle inequality of the norm, we have 
\[\lVert\Pi_h \widetilde{u} \rVert_{L _p^{2}\left(\Gamma; L^{2}(D)\right)} \leq c \lVert \widetilde{u} \rVert_{L _p^{2}\left(\Gamma; H^{2}(D)\right)} \text{ and } \lVert\Pi_h \widetilde{u} \rVert_{L_p^{2}\left(\Gamma; H_0^{1}(D)\right)}  \leq c \lVert \widetilde{u} \rVert_{L _p^{2}\left(\Gamma; H^{2}(D)\right)}.\]
Due to \(\partial_y^\alpha\!\big(\Pi_h \widetilde u(\cdot,y)\big) = \Pi_h\!\big(\partial_y^\alpha \widetilde u(\cdot,y)\big)
\  \text{for all multi-indices } \alpha\), we have
\[\|\Pi_h \widetilde u\|_{H_p^2(\Gamma;L^2(D))}^2
 = \sum_{|\alpha|\le 2}\int_\Gamma p(y)\,
\big\|\Pi_h\big(\partial_y^\alpha \widetilde u\big)\big\|_{L^2(D)}^{2}\,dy 
 \leq  c \sum_{|\alpha|\le 2} \lVert \partial_y^\alpha \widetilde u\rVert _{L_p^2(\Gamma;H^2(D))}^2
 = c \lVert \widetilde{u} \rVert _{H_p^2(\Gamma;H^2(D))}^2.\]

 Then 
\[\begin{aligned}
	\lVert \Pi_s \Pi_h  \widetilde{u} -  \widetilde{u}\rVert _{L _p^{2}\left(\Gamma; L^{2}(D)\right)}  &   \leq \lVert \Pi_s \Pi_h  \widetilde{u} -  \Pi_h \widetilde{u}\rVert_{L _p^{2}\left(\Gamma; L^{2}(D)\right)} + \lVert \Pi_h \widetilde{u} -  \widetilde{u}\rVert _{L _p^{2}\left(\Gamma; L^{2}(D)\right)}\\
	& \leq c s^2  \lvert \Pi_h \widetilde{u} \rvert_{H_p^2(\Gamma;L^2(D))} + c h^{2} \, |\widetilde{u}|_{L _p^{2}\left(\Gamma; H^{2}(D)\right)}\\
	& \leq cs^2 \lVert \widetilde{u} \rVert_{H_p^2(\Gamma; H^2(D))} + ch^2 \lVert \widetilde{u} \rVert_{L_p^2(\Gamma;H^2(D))}\\
	& \leq c(s^2 + h^2)\lVert \widetilde{u} \rVert _{H_p^2(\Gamma;H^2(D))}.
\end{aligned}
\]

Similary, we can get 
\[\begin{aligned}
	&\lVert \Pi_s \Pi_h  \widetilde{g} -  \widetilde{g}\rVert _{L _p^{2}\left(\Gamma; L^{2}(D)\right)} \leq c(s^2 + h^2) \lVert \widetilde{g} \rVert_{H_p^2(\Gamma;H^2(D))},\\
	&\lVert \Pi_s \Pi_h  \widetilde{u} -  \widetilde{u}\rVert _{L _p^{2}\left(\Gamma; H_0^1(D)\right)}  \leq c(s^2 + h) \lVert \widetilde{u} \rVert_{H_p^2(\Gamma;H^2(D))}.
\end{aligned}\]

Hence,
\begin{align*}
	\inf _{v^{hs} \in \mathcal{K}^{hs}} &  \left[\left\|    \widetilde{u}-v^{hs}\right\|_{W}^{2} +   \left\|-\operatorname{div}(\widetilde{a}\, \nabla \widetilde{u})-\widetilde{f}\right\|_{L_p^2(\Gamma;L^{2}(D))}\left\|v^{hs}-\widetilde{u}\right\|_{L_p^2(\Gamma;L^{2}(D))} \right]\\
	& \leq \left\|  \widetilde{u}- \Pi_{s}\Pi_{h} \widetilde{u}\right\|_{W}^{2} +   \left\|-\operatorname{div}(\widetilde{a}\, \nabla \widetilde{u})-\widetilde{f}\right\|_{L_p^2(\Gamma;L^{2}(D))}\left\| \Pi_{s} \Pi_{h} \widetilde{u}-\widetilde{u}\right\|_{L_p^2(\Gamma;L^{2}(D))}\\
	& \leq c(h^2 + s^2).
\end{align*}

Following the idea of Falk (\cite{falk1974error}), we bound the $\inf _{v \in \mathcal{K}} \left\|v-\widetilde{u}^{hs} \right\|_{L_p^2(\Gamma;L^{2}(D))}$.

Denote 
\[
	\widetilde{u}^{hs,*}(\bx,\by) = \max \{\widetilde{u}^{hs}(\bx,\by), \widetilde{g}(\bx,\by)\},\quad  (D \times \Gamma)^{*} = \{\bx \in D, \by \in \Gamma \mid \widetilde{u}^{hs}(\bx, \by) < \widetilde{g}(\bx, \by)\},
\]
Then, over { $ D \times \Gamma \setminus  (D \times \Gamma)^{*} $}, $ \widetilde{u}^{hs,*} =  \widetilde{u}^{hs} $. Since $\widetilde{u}^{hs} \in L_p^2(\Gamma;H_0^1(D))$ and $\widetilde{g} \in L_p^2(\Gamma;H^2(D))$ with $\widetilde{g} \leq 0$ on $\partial D \times \Gamma$, clearly, $\widetilde{u}^{hs,*} \in \mathcal{K}$,
\[
	\inf _{v \in \mathcal{K}} \left\|v-\widetilde{u}^{hs} \right\|_{L_p^2(\Gamma;L^{2}(D))} ^2 \leq  \left\| \widetilde{u}^{hs,*}-\widetilde{u}^{hs} \right\|_{L_p^2(\Gamma;L^{2}(D))}^2  = \int_  {(D\times \Gamma)^{*}} | \widetilde{u}^{hs,*}(\bx,\by)-\widetilde{u}^{hs}(\bx,\by) |^2  p(\by) \, \mathrm{d}  \bx \, \mathrm{d} \by. 
\]
Let $ \Pi_s \Pi_h \widetilde{g}$ be the continuous piecewise linear interpolant of $\widetilde{g}$. 
Since the discrete solution satisfies the nodal constraint, at any node $(\bx_i, \by_j)$, $ \widetilde{u}^{hs}(\bx_i, \by_j) \geq \widetilde{g}(\bx_i, \by_j)$, we have $\widetilde{u}^{hs} \geq \Pi_s \Pi _h \widetilde{g}$ in $D \times \Gamma$. 
Hence, over $(D \times \Gamma)^{*}$, we have 
\[
	0 < | \widetilde{u}^{hs}(\bx,\by) - \widetilde{g}(\bx,\by)|  = \widetilde{g}(\bx,\by) - \widetilde{u}^{hs}(\bx,\by) \leq \widetilde{g}(\bx,\by) - \Pi_s \Pi _h \widetilde{g}(\bx,\by).
\]
Thus, we get
\begin{align*}
	\int_  {(D \times \Gamma)^{*}} | \widetilde{u}^{hs,*}(\bx,\by)-\widetilde{u}^{hs}(\bx,\by)  |^2 p(\by) \,\mathrm{d}  \bx \, \mathrm{d}  \by  &\leq \int_  {(D \times \Gamma )^{*}}\,| \widetilde{g}(\bx,\by) - \Pi_s \Pi _h \widetilde{g} (\bx,\by) |^2 p(\by) \, \mathrm{d}  \bx \, \mathrm{d} \by \\
	& \leq \int_  {D \times \Gamma} | \widetilde{g}(\bx,\by) - \Pi_s \Pi _h \widetilde{g}(\bx,\by) |^2  p(\by) \,\mathrm{d}  \bx \, \mathrm{d} \by  \\
	& = \lVert \widetilde{g} - \Pi_s \Pi_h \widetilde{g}  \rVert_{L_p^2(\Gamma; L^2(D))}^2.
\end{align*}
Then 
\[
	\inf _{v \in \mathcal{K}} \left\|v-\widetilde{u}^{hs} \right\|_{L_p^2(\Gamma;L^{2}(D))}  \leq  c (s^2 + h^2) \,|\widetilde{g}|_{H_p^{2}\left(\Gamma; H^{2}(D)\right)}.
\]
Finally, we get the order estimate
\[
	\left\|\widetilde{u}-\widetilde{u}^{hs}\right\|_{W} \leq  c (h + s),
\]
where constant $c > 0$ depending only on   $\lVert \widetilde{u} \rVert _{L _p^{2}\left(\Gamma; H^{2}(D)\right)}$,
$\|\widetilde{f}\|_{H_p^2(\Gamma;L^{2}(D))}$ and $ \lVert \widetilde{g} \rVert_{H_p^{2}\left(\Gamma; H^{2}(D)\right)}$.
\end{proof}
}

In conclusion, under the assumptions made in Theorem~\ref{thm4.3}, for the numerical solution using linear element,
\[ \lVert \widetilde{u} - \widetilde{u}^{hs}\rVert _W \leq c (h+s),\]
where $\widetilde{u}$ and $\widetilde u^{hs}$ are  the solutions of (\ref{eq11}) and (\ref{eq4.1}). 

{ Then, by Jensen's inequality \cite[Theorem 1.6.2]{Durrett2019Probability} we obtain
\begin{align*}
	\lVert \mathbb{E}[\widetilde{u}] - \mathbb{E} [\widetilde{u}^{hs}] \rVert _{H_0^1(D)} ^2  & \leq \mathbb{E} \left[\lVert \widetilde{u} - \widetilde{u}^{hs} \rVert_{H_0^1(D)}^2 \right].
\end{align*}
Using the definition \(\|v\|_{W}^{2}=\mathbb{E}\!\left[\|v\|_{H_0^1(D)}^{2}\right]\), we conclude that
\[\lVert \mathbb{E}[\widetilde{u}] - \mathbb{E} [\widetilde{u}^{hs}] \rVert _{H_0^1(D)} ^2  \leq \lVert \widetilde{u} - \widetilde{u}^{hs} \rVert _W^2. \]}

Taking the square root on both sides,
\begin{align*}
	\lVert \mathbb{E}[\widetilde{u}] - \mathbb{E} [\widetilde{u}^{hs}] \rVert _{H_0^1(D)}  & \leq  \lVert \widetilde{u} - \widetilde{u}^{hs} \rVert _W\\
	& \leq \lVert \widetilde{u} - \widetilde{u}^{h} \rVert _W + \lVert \widetilde{u}^{h} - \widetilde{u}^{hs} \rVert _W\\
	& \leq c(h+s).
\end{align*}
Similarly, we can give the approximation of second moment as in \cite{Babuska2007},
\begin{align*}
	\lVert \mathbb{E}[(\widetilde{u})^2 - (\widetilde{u}^{hs})^2] \rVert_{H_0^1(D)} & = \lVert \mathbb{E}[(\widetilde{u} - \widetilde{u}^{hs})(\widetilde{u} + \widetilde{u}^{hs})] \rVert_{H_0^1(D)} \\
		& =  \lVert \mathbb{E}[(\widetilde{u} - \widetilde{u}^{hs})(2\widetilde{u} -(\widetilde{u} - \widetilde{u}^{hs}))] \rVert_{H_0^1(D)}\\
		& = \lVert 2\mathbb{E}[\widetilde{u}(\widetilde{u} - \widetilde{u}^{hs})] + \mathbb{E}[ (\widetilde{u} - \widetilde{u}^{hs})^2] \rVert_{H_0^1(D)}\\
		& \leq 2 \lVert \widetilde{u} \rVert _W \lVert \widetilde{u} - \widetilde{u}^{hs} \rVert _W +   \lVert \widetilde{u} - \widetilde{u}^{hs} \rVert _W ^2\\
		& \leq c \lVert \widetilde{u} - \widetilde{u}^{hs} \rVert _W  \\
		& \leq c(h+s).
\end{align*}

\subsection{Stochastic Galerkin system}
{ 
We now discuss the nonlinear algebraic system that must be solved to compute mean and variance following the notation of~\cite[page 405]{Lord2014introduction}. Our approximate space is the tensor-product space  $W^{hs} = X^h \otimes Y^s$, where $X^h$ is the finite element space of dimension $I$ and $Y^s$ is the piecewise polynomial space of dimension $J$. Consequently, the total number of degrees of freedom is $IJ$.

Starting from the fully discrete formulation~\eqref{eq4.1},
\begin{align*}
\widetilde{b}(u, v) & =\int_{\Gamma} p(\boldsymbol{y}) \int_{D} \widetilde{a}(\boldsymbol{x}, \boldsymbol{y}) \nabla u(\boldsymbol{x}, \boldsymbol{y}) \cdot \nabla v(\boldsymbol{x}, \boldsymbol{y}) \, \mathrm{d} \bx\, \mathrm{d} \by , \\
\widetilde{\ell}(v) & =\int_{\Gamma} p(\boldsymbol{y}) \int_{D} \widetilde{f}(\boldsymbol{x}, \boldsymbol{y}) v(\boldsymbol{x}, \boldsymbol{y}) \, \mathrm{d} \bx\, \mathrm{d} \by.
\end{align*}
The parameter fields are represented by the truncated Karhunen-Lo\`eve expansion \eqref{truncate_a} and \eqref{truncate_f}
\[ \widetilde{a}(\boldsymbol{x}, \boldsymbol{y}) = \mu_{a}(\boldsymbol{x})+\sum_{k=1}^{M_a} \sqrt{v_{k}^{a}} \phi_{k}^{a}(\boldsymbol{x}) y_{k}, \quad
\widetilde{f}(\boldsymbol{x}, \boldsymbol{y}) = \mu_{f}(\boldsymbol{x})+\sum_{k=1}^{M_f} \sqrt{v_{k}^{f}} \phi_{k}^{f}(\boldsymbol{x}) y_k.
\]
We seek $\widetilde u^{hs} \in \widetilde{\mathcal{K}}^{hs}$ such that
\[\widetilde u^{hs}\in \widetilde{\mathcal{K}}^{hs}, \quad \widetilde b( \widetilde u^{hs}, v-\widetilde u^{hs}) \geq  \widetilde \ell(v- \widetilde u^{hs}), \quad \forall v \in \widetilde{\mathcal{K}}^{hs}.
\]

By expanding  $\widetilde{u}^{hs}$  as in~\eqref{eq4.1},  $\widetilde{u}^{hs}(\boldsymbol{x}, \boldsymbol{y}) = \sum_{i=1}^{I} \sum_{j=1}^{J} {u}_{ij} \phi_{i}(\boldsymbol{x}) \psi_{j}(\boldsymbol{y})$
and testing against all basis functions $v^{hs} = \phi_{r} \psi_{t}$ for  $r=1, \ldots, I$  and  $t=1, \ldots, J$. This process yields a matrix $A \in \mathbb{R}^{IJ \times IJ}$ representing the discretized bilinear form and a vector $\boldsymbol{b} \in \mathbb{R}^{IJ}$ representing the discretized linear form:
\[\begin{aligned}
    [A]_{(r,t),(i,j)} &= \tilde{b}(\phi_i\psi_j, \phi_r\psi_t), \\
    [\mathbf{b}]_{(r,t)} &= \tilde{\ell}(\phi_r\psi_t),
\end{aligned}
 \quad \text{for } r, i = 1, \ldots, I \text{ and } t, j = 1, \ldots, J.\]
This results in a block-structured matrix and corresponding vectors: 
\[
	A=\left(\begin{array}{cccc}
		A_{11} & A_{12} & \ldots & A_{1 J} \\
		A_{21} & A_{22} & \ldots & A_{2 J} \\
		\vdots & \vdots & \ddots & \vdots \\
		A_{J 1} & A_{J 2} & \ldots & A_{J J}
		\end{array}\right), \quad \boldsymbol{u}=\left(\begin{array}{c}
		\boldsymbol{u}_{1} \\
		\boldsymbol{u}_{2} \\
		\vdots \\
		\boldsymbol{u}_{J}
		\end{array}\right),
		\quad \boldsymbol{b}=\left(\begin{array}{c}
		\boldsymbol{b}_{1} \\
		\boldsymbol{b}_{2} \\
		\vdots \\ 
		\boldsymbol{b}_{J}
		\end{array}\right).
\]
The  $j^{\rm th}$ block of the solution vector is given by 
\[
	\boldsymbol{u}_{j} = \left[u_{1 j}, u_{2 j}, \ldots, u_{I j}\right]^{\top}, \quad j=1, \ldots, J.
\]

The constituent blocks are given by:
\[\begin{aligned}
	A_{t j} &= \left\langle\psi_{j}, \psi_{t}\right\rangle_{p} K_{0}+\sum_{k=1}^{M_a}\left\langle y_{k} \psi_{j}, \psi_{t}\right\rangle_{p} K_{k},\\
	\boldsymbol{b}_{t} &= \left\langle \psi_{t},1\right\rangle_{p}\boldsymbol{f}_{0}+\sum_{k=1}^{M_f}\left\langle y_{k}, \psi_{t}\right\rangle_{p} \boldsymbol{f}_{k},
\end{aligned}
\quad \text{for}\  t, j = 1, \ldots, J.
	\]
The finite element matrices $K_{0}, K_k \in \mathbb{R}^{I\times I}$ and vector $\boldsymbol{f}_{0}, \boldsymbol{f}_{k} \in \mathbb{R}^I$ are defined as:
\[
	\begin{aligned}
		\left[K_{0}\right]_{r i} & =  \int_{D} \mu_{a}(\boldsymbol{x}) \nabla \phi_{i}(\boldsymbol{x}) \cdot \nabla \phi_{r}(\boldsymbol{x}) \, \mathrm{d} \bx,  
		\quad
		&\left[K_{k}\right]_{r i} & = \int_{D}\left(\sqrt{v_{k}^{a}} \phi_{k}^{a}(\boldsymbol{x})\right) \nabla \phi_{i}(\boldsymbol{x}) \cdot \nabla \phi_{r}(\boldsymbol{x})\, \mathrm{d} \bx,\\
		\left[\boldsymbol{f}_{0}\right]_{r}  & = \int_{D} \mu_{f}(\boldsymbol{x}) \phi_{r}(\boldsymbol{x})\,\mathrm{d} \bx,  
		\quad
		&\left[\boldsymbol{f}_{k}\right]_{r}  & = \int_{D}\left(\sqrt{v_{k}^{f}} \phi_{k}^{f}(\boldsymbol{x})\right) \phi_{r}(\boldsymbol{x}) \,\mathrm{d} \bx.
		\end{aligned}
\]

This block structure admits a compact Kronecker product representation (\cite{Powell2009,Ernst2009}):
\[\begin{aligned}
	&A = G_{0} \otimes K_{0}+\sum_{k=1}^{M_a} G_{k} \otimes K_{k}, \\
	&\boldsymbol{b} = \boldsymbol{g}_{0} \otimes \boldsymbol{f}_{0} +\sum_{k=1}^{M_f} \boldsymbol{g}_{k} \otimes \boldsymbol{f}_{k},
	\end{aligned}
	\]
where the matrices $G_{0},G_k \in \mathbb{R}^{J \times J}$ and vectors  $\boldsymbol{g}_{0}, \boldsymbol{g}_{k} \in \mathbb{R}^J$ are defined by:
\[
	\begin{aligned}
		\left[G_{0}\right]_{j t}&=\left\langle\psi_{j}, \psi_{t}\right\rangle_{p},   \quad
		&\left[G_{k}\right]_{j t}&=\left\langle y_{k} \psi_{j}, \psi_{t}\right\rangle_{p},\\
		\left[\boldsymbol{g}_{0}\right]_{t} &= \left\langle\psi_{t},{1}\right\rangle_{p}, \quad
		&\left[\boldsymbol{g}_{k}\right]_{t} &= \left\langle y_{k},\psi_{t}\right\rangle_{p}.
		\end{aligned}
\]

Crucially, the discretization of the variational inequality \eqref{eq4.1} is not the linear system \( A\mathbf{u} = \mathbf{b} \), but rather the nonlinear complementarity problem:
\[
\mathbf{u} \geq \mathbf{g}, \quad A\mathbf{u} - \mathbf{b} \geq \mathbf{0}, \quad (\mathbf{u} - \mathbf{g})^{\top}(A\mathbf{u} - \mathbf{b}) = 0,
\]
where \( \mathbf{g} \in \mathbb{R}^{IJ} \) is the vector of nodal values of the obstacle, \( g_{ij} = \widetilde{g}(\bx_i, \by_j)\) and the last condition implies component-wise complementarity:
\[
(u_{ij} - g_{ij})[A\mathbf{u} - \mathbf{b}]_{ij} = 0 \quad i=1,\ldots,I \text{ and } j = 1, \ldots, J.
\] }

When adapting the active set method (\cite{Hueber2005, Karkkainen2003}) to solve the variational inequality (\ref{eq4.1}), {at any node \((\boldsymbol{x}_{i'}, \boldsymbol{y}_{j'})\)}, we need only satisfy the condition:
\[
\widetilde{u}^{hs}(\boldsymbol{x}_{i'}, \boldsymbol{y}_{j'}) = \sum_{i=1}^{I} \sum_{j=1}^{J} u_{ij} \phi_{i}(\boldsymbol{x}_{i'}) \psi_{j}(\boldsymbol{y}_{j'}) \geq g(\bx_{i'}, \by_{j'} ),
\]
which is equivalent to the condition \(u_{ij} \geq g_{ij}\) for \(i = 1, \dots, I\) and \(j = 1, \dots, J\).

\section{Numerical Experiments}\label{sec5}
\setcounter{equation}0

In applications, specific statistics and moments of  \(\widetilde{u}^{hs}\) are of interest, and this is also the primary goal of our computations.

Recall that \[L_p^2(\Gamma) \coloneqq \{v\colon\Gamma\to \mathbb{R} \mid \lVert v \rVert _{L_p^2(\Gamma)}<\infty\}, \quad \lVert v \rVert _{L_p^2(\Gamma)}^2 \coloneqq \langle v,v \rangle _p, \]
and 
\begin{equation}\label{notation}
\langle v,w \rangle _p \coloneqq \int_\Gamma p(\boldsymbol y)v(\boldsymbol y)w(\boldsymbol y)\,\mathrm{d} \boldsymbol y.
\end{equation}

Let $X^{h}\subset H_{0}^{1}(D)$ denote any  $I$-dimensional finite element subspace with
\[
	X^{h}=\operatorname{span}\left\{\phi_{1}, \phi_{2}, \ldots, \phi_{I}\right\}
\]
and $Y^{s} \subset   L_{p}^{2}(\Gamma)$  with
\[
	Y^{s}=\operatorname{span}\left\{\psi_{1}, \psi_{2}, \ldots, \psi_{J}\right\}.
\]

The tensor product space is defined as
\begin{equation}
	\label{eq5.2}
	X^{h} \otimes Y^{s} \coloneqq \operatorname{span}\left\{\phi_{i} \psi_{j}\colon i=1, \ldots, I, j=1, \ldots, J \right\}
\end{equation}
and it has dimension $IJ$. 

Any function $w \in W^{hs}$ is of the form
\begin{equation}
	w(\boldsymbol{x}, \boldsymbol{y}) = \sum_{i=1}^{I} \sum_{j=1}^{J} w_{i j} \phi_{i}(\boldsymbol{x}) \psi_{j}(\boldsymbol{y}).
	\label{eq16}
\end{equation}
In particular, write
\begin{align*}
\widetilde{u}^{hs}(\boldsymbol{x}, \boldsymbol{y}) &= \sum_{i=1}^{I} \sum_{j=1}^{J} u_{i j} \phi_{i}(\boldsymbol{x}) \psi_{j}(\boldsymbol{y}) = \sum_{j = 1}^{J} \left(\sum_{i = 1}^{I} u_{ij}\phi _i (\boldsymbol{x})\right)\psi_j(\boldsymbol{y})\\
  & = \sum_{j = 1}^{J} \widetilde{u}_{j}(\boldsymbol{x})  \psi_j(\boldsymbol{y}),
\end{align*}
where the coefficients  $u_{i j}$  are found by solving a nonlinear system of dimension  $IJ$ {due to the complementarity condition.} Then with the notation (\ref{notation}), we get
\begin{align*}
	\mathbb{E}\left[\widetilde{u}^{hs}\right] &= \int_{\Gamma}p(\boldsymbol{y})\widetilde{u}^{hs}(\boldsymbol{x}, \boldsymbol{y}) d \boldsymbol{y} = \sum_{i=1}^{I} \sum_{j=1}^{J} u_{i j} \phi_{i}(\boldsymbol{x}) \int_{\Gamma}p(\boldsymbol{y}) \psi_{j}(\boldsymbol{y}) \,\mathrm{d}\boldsymbol{y} \\
& =  \sum_{j=1}^{J} \widetilde{u}_j(\boldsymbol{x}) \int_{\Gamma} p(\boldsymbol{y})\psi_j(\boldsymbol{y})\,\mathrm{d} \boldsymbol{y} = \sum_{j=1}^{J}\widetilde{u}_j \left\langle\psi_j, 1 \right\rangle_{p}.
\end{align*}

Similarly, the variance is
\[
	\operatorname{Var}\left(\widetilde{u}^{hs}\right)=\mathbb{E}\left[({\widetilde{u}^{hs}})^{2}\right]-\mathbb{E}\left[\widetilde{u}^{hs}\right]^{2}.
\]
Expanding $ \widetilde{u}^{hs}$  gives
\[
	\begin{aligned}
		\mathbb{E}\left[(\widetilde{u}^{hs})^{2}\right] & =\int_{\Gamma} p(\boldsymbol{y}) \widetilde{u}^{hs}(\boldsymbol{x}, \boldsymbol{y})^{2} \,\mathrm{d} \boldsymbol{y}=\left\langle\widetilde{u}^{hs}, \widetilde{u}^{hs}\right\rangle_{p} \\
		 & = \sum_{k = 1}^{J} \sum_{j=1}^{J} \widetilde{u}_k \widetilde{u}_j \left\langle \psi_k,\psi_j\right\rangle_{p}.
		\end{aligned}
\]
Hence, we have
\[
	{\operatorname{Var}\left(\widetilde{u}^{hs}\right)}= \sum_{k = 1}^{J} \sum_{j=1}^{J} \widetilde{u}_k \widetilde{u}_j (\left\langle \psi_k,\psi_j\right\rangle_{p} - \left\langle \psi_k,1\right\rangle_{p}\left\langle \psi_j,1\right\rangle_{p}).
\]

To demonstrate the effectiveness of the proposed method and to verify the theoretical convergence rates, we present two numerical examples. 
Example~\ref{exm1} addresses a model with a stochastic diffusion coefficient, whereas Example~\ref{exm2} deals with a stochastic source term. 
These examples are designed to test the method under different types of stochasticity. The numerical results in both cases show excellent agreement with the theoretical predictions.

\begin{example}
	\label{exm1}
Let $ D=(-1.5,1.5)^2$ and 

\[ a(x_1,x_2, \omega) = 1 + \exp(\xi_1) + 2\exp(\xi_2),
\]
\[{f}(x_1,x_2,\omega) = -2,
\]
where $\xi_1, \xi_2 \thicksim \text{i.i.d.} \, \mathrm{U}(- 1, 1)$, and the obstacle function $g = 0$. The exact solution is
\[
	u(x_1,x_2,\xi_1,\xi_2) = \begin{cases}
		0, & x_1^2 + x_2^2 \leq 1, \\
	(\frac{x_1^2+x_2^2}{2}-\frac{\log(x_1^2+x_2^2)}{2}-\frac{1}{2})* \frac{1}{1+\exp({\xi_1})+2\exp({\xi_2})}, & x_1^2+x_2^2 >1.
	\end{cases}
\]
\end{example}

Tables~\ref{table1} and~\ref{table3} present the $L^2(D)$-norm and $H^1(D)$-norm relative errors between the expected values of the solution $u$ and its numerical approximation $u^{hs}$, computed as
\[e_{L^2}^{rel,1}(u,u^{hs}) =  \frac{\big\| \mathbb{E}[u] - \mathbb{E}[u^{hs}] \big\|_{L^2(D)}}{\big\| \mathbb{E}[u] \big\|_{L^2(D)}} \quad \text{and} \quad  e_{H^1}^{rel,1}(u,u^{hs}) =  \frac{\big\| \nabla \mathbb{E}[u] - \nabla \mathbb{E}[u^{hs}] \big\|_{L^2(D)}}{\big\| \nabla \mathbb{E}[u] \big\|_{L^2(D)}}.\]

Similarly, Table~\ref{table2} presents the corresponding relative errors for the second moments, evaluated as
\[ e_{L^2}^{rel,2}(u,u^{hs}) =  \frac{\big\| \mathbb{E}[{(u)}^2] - \mathbb{E}[{(u^{hs})}^2] \big\|_{L^2(D)}}{\big\| \mathbb{E}[{(u)}^2] \big\|_{L^2(D)}} \quad \text{and} \quad  e_{H^1}^{rel,2}(u,u^{hs}) = \frac{ \big\| \nabla \mathbb{E}[{(u)}^2] - \nabla \mathbb{E}[{(\widetilde{u}^{hs})}^2] \big\|_{L^2(D)}}{ \big\| \nabla \mathbb{E}[{(\widetilde{u})}^2] \big\|_{L^2(D)}}.\]

\begin{table}[!h]
		\centering
\caption{ Relative errors of expected values with $s = (e - e^{-1})/2^{4}$ } 
		\begin{tabular}{ |c|c|c|c|c| } 
			\hline
			$h$		  & $e_{L^2}^{rel,1}(u,u^{hs})$ & order & $e_{H^1}^{rel,1}(u,u^{hs})$ & order \\ \hline
			$3/2^{2}$ &3.6004e-01 &      &4.7088e-01&       \\ \hline
			$3/2^{3}$ &7.6845e-02 &2.2281&2.3794e-01&0.9847 \\ \hline
			$3/2^{4}$ &2.1563e-02 &1.8333&1.2199e-01&0.9638 \\ \hline
			$3/2^{5}$ &5.9221e-03 &1.8644&6.1089e-02&0.9979 \\ 
			\hline
		   \end{tabular}
		   \label{table1}
	\end{table}
	\begin{table}[!h]
		\centering
		\caption{Relative errors of second moments with $s  = (e-e^{-1})/2^4$}
		\begin{tabular}{ |c|c|c|c|c| } 
			\hline
			$h$		  & $e_{L^2}^{rel,2}(u,u^{hs})$ & order   &  $e_{H^1}^{rel,2}(u,u^{hs})$  & order   \\ 
			\hline
			$3/2^{2}$ &8.0635e-01&	    &7.6605e-01&		 \\ \hline
			$3/2^{3}$ &2.2570e-01&1.8370&4.2420e-01&0.8527  \\ \hline
			$3/2^{4}$ &5.9037e-02&1.9347&2.1764e-01&0.9628 \\ \hline
			$3/2^{5}$ &1.6042e-02&1.8797&1.0952e-01&0.9907 \\ 
			\hline
		   \end{tabular}
		   \label{table2}
	\end{table}  

From Tables~\ref{table1} and \ref{table2}, we can see the convergence orders in the $H^1$-norm  are about 1, which is consistent with our theoretical results.
\begin{table}[!h]
		\centering
		\caption{Relative errors of expected values with $h = \frac{3}{2(e-e^{-1})}s$ }
		\begin{tabular}{ |c|c|c|c|c|c|c| } 
			\hline
			$s$ & $e_{L^2}^{rel,1}(u,u^{hs})$ & order & $e_{H^1}^{rel,1}(u,u^{hs})$ & order \\ 
			\hline
			$(e-e^{-1})/2^{0}$ &4.1212e-01&      &4.7822e-01&      \\ \hline
			$(e-e^{-1})/2^{1}$ &8.9613e-02&2.2013&2.3942e-01&0.9981\\ \hline
			$(e-e^{-1})/2^{2}$ &2.4148e-02&1.8918&1.2207e-01&0.9717\\ \hline
			$(e-e^{-1})/2^{3}$ &5.9221e-03&2.0277&6.1089e-02&0.9988\\ 
			\hline
		   \end{tabular}
		   \label{table3}
	\end{table}

Since the error in the probability space is not significant, it is difficult to observe the convergence order with respect to \( s \) when \( h \) is fixed, even at \( h = {2^{-5}} \). To mitigate this issue, we adopt a compromise by setting \( h = \frac{3}{2(e - e^{-1})} s \), aiming to better capture the convergence behavior with respect to \( s \).

\begin{figure}[htbp]
    \centering
    \begin{subfigure}{0.32\textwidth}
        \includegraphics[width=\linewidth]{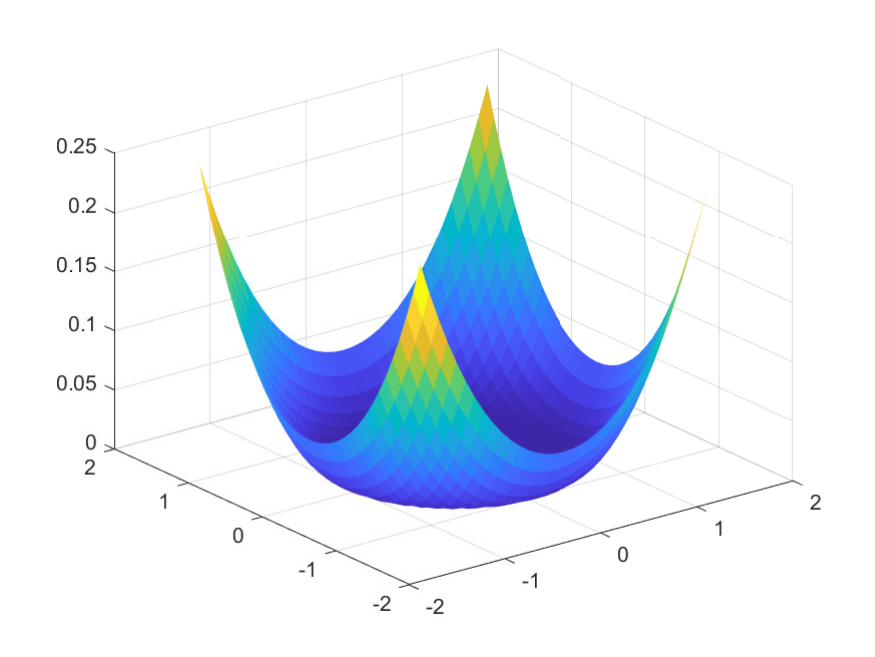}
        \caption{Expectation of the Exact solution }
    \end{subfigure}
    \hfill
    \begin{subfigure}{0.32\textwidth}
        \includegraphics[width=\linewidth]{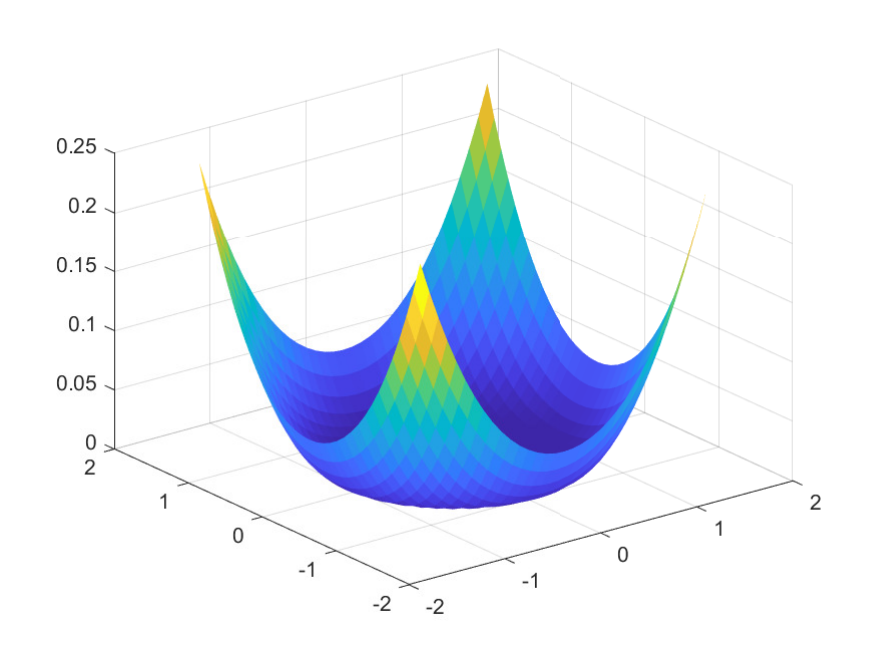}
        \caption{ Expectation of the SG solution}
    \end{subfigure}
    \hfill
    \begin{subfigure}{0.32\textwidth}
        \includegraphics[width=\linewidth]{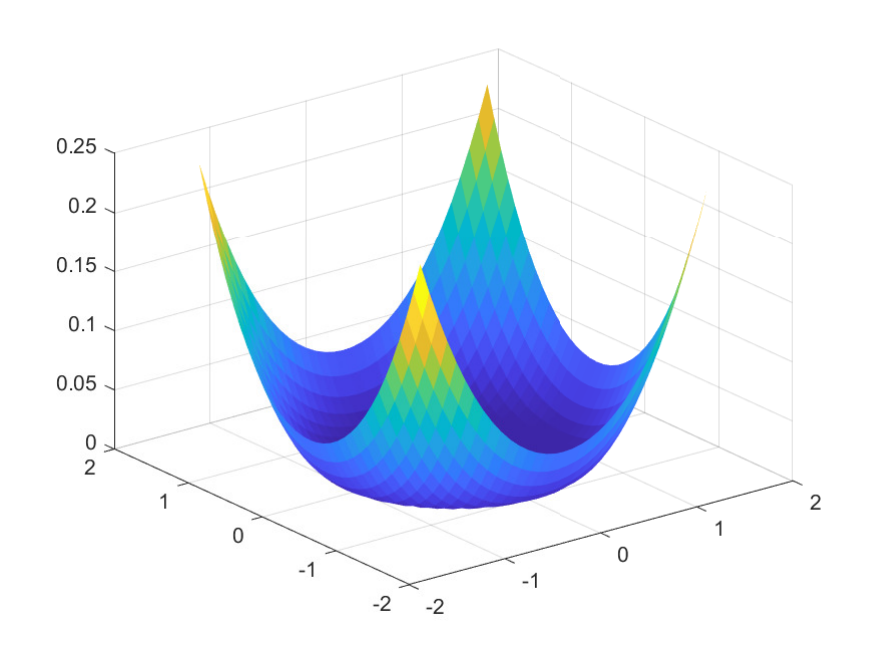}
        \caption{Mean of the MC solution \quad \quad \qquad \quad \quad \qquad   }
    \end{subfigure}

    \vspace{1em}

    \begin{subfigure}{0.32\textwidth}
        \includegraphics[width=\linewidth]{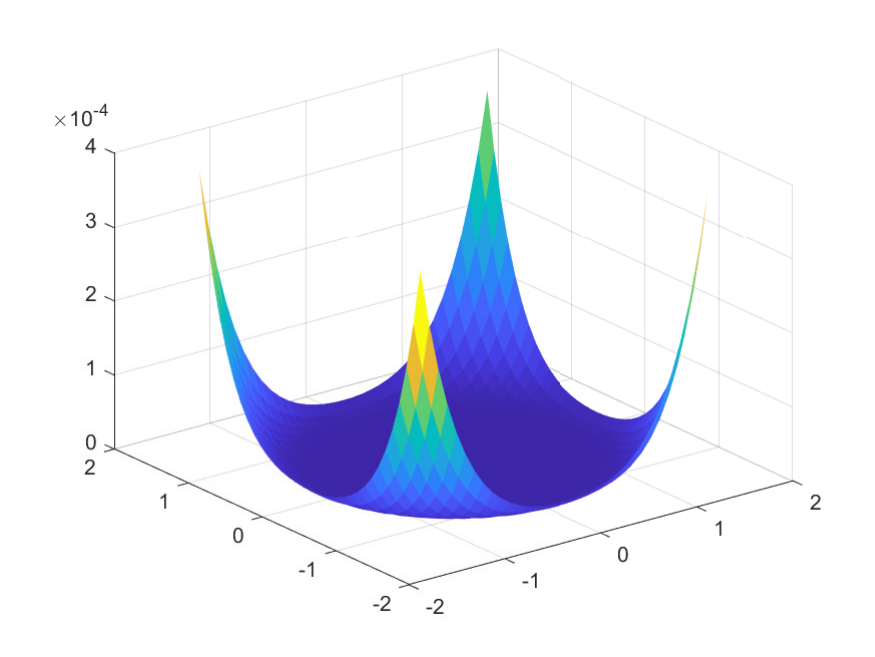}
        \caption{ Variance of the Exact solution  }
    \end{subfigure}
    \hfill
    \begin{subfigure}{0.32\textwidth}
        \includegraphics[width=\linewidth]{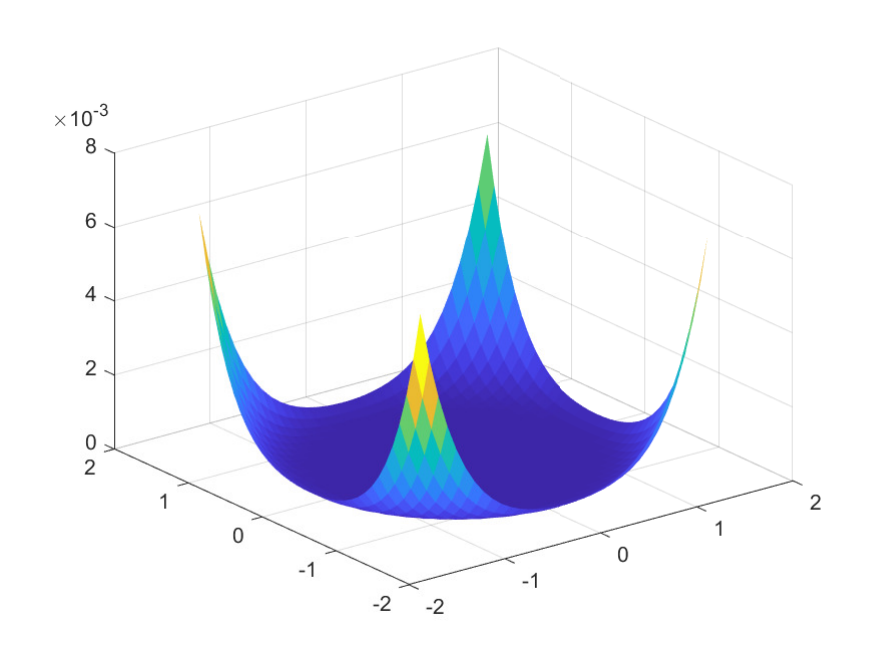}
        \caption{Variance of the SG solution  }
    \end{subfigure}
    \hfill
    \begin{subfigure}{0.32\textwidth}
        \includegraphics[width=\linewidth]{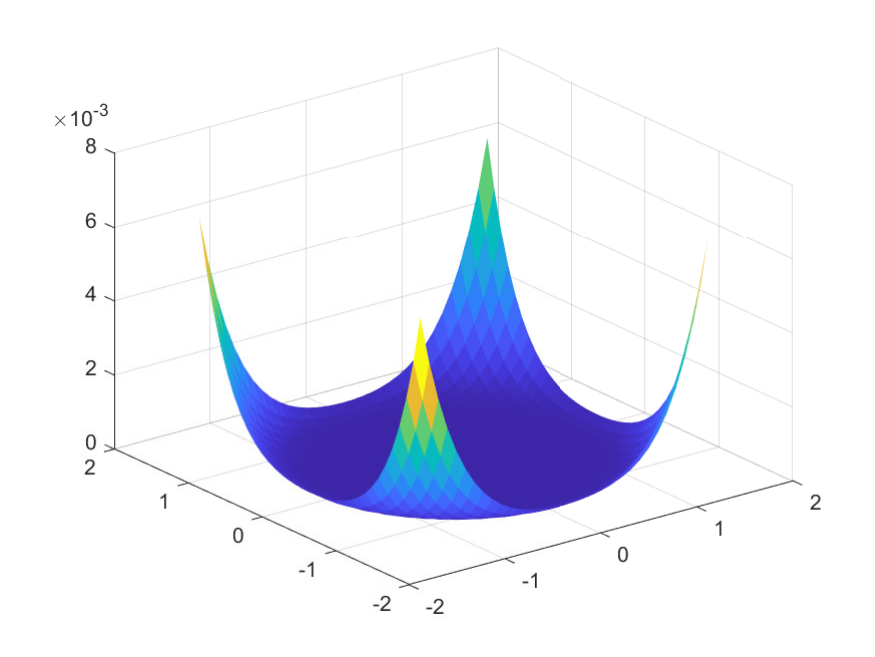}
        \caption{Variance of the MC solution}
    \end{subfigure}

    \caption{Comparison of statistical properties: expectation and variance of the exact solution, SG solution, and MC solution.}
	\label{fig1}
\end{figure}

\begin{figure}[htbp]
    \centering
    \begin{subfigure}{0.45\textwidth}
        \includegraphics[width=\linewidth]{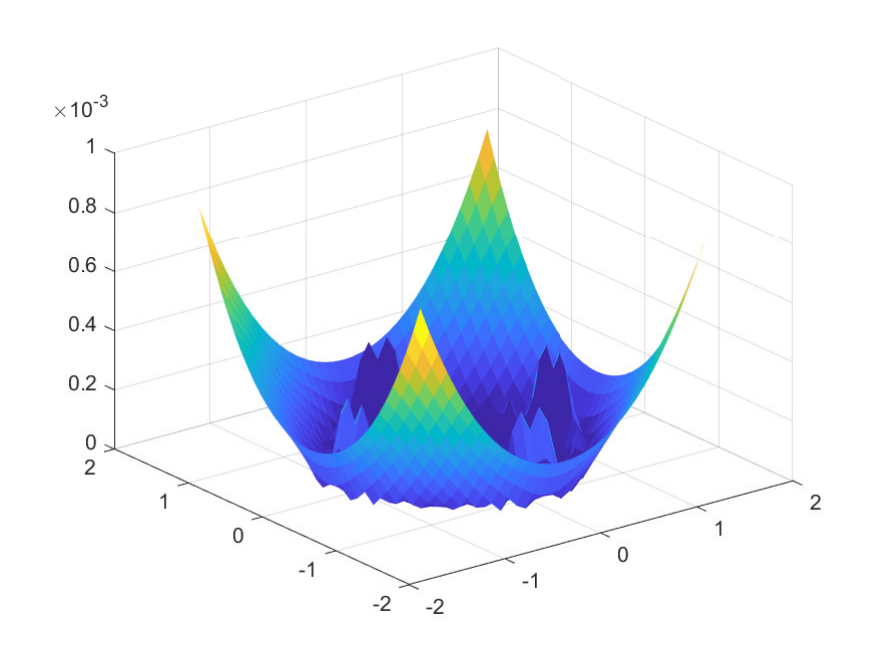}
        \caption{Error between expectation of exact  solution and SG solution}
    \end{subfigure}
    \hfill
    \begin{subfigure}{0.45\textwidth}
        \includegraphics[width=\linewidth]{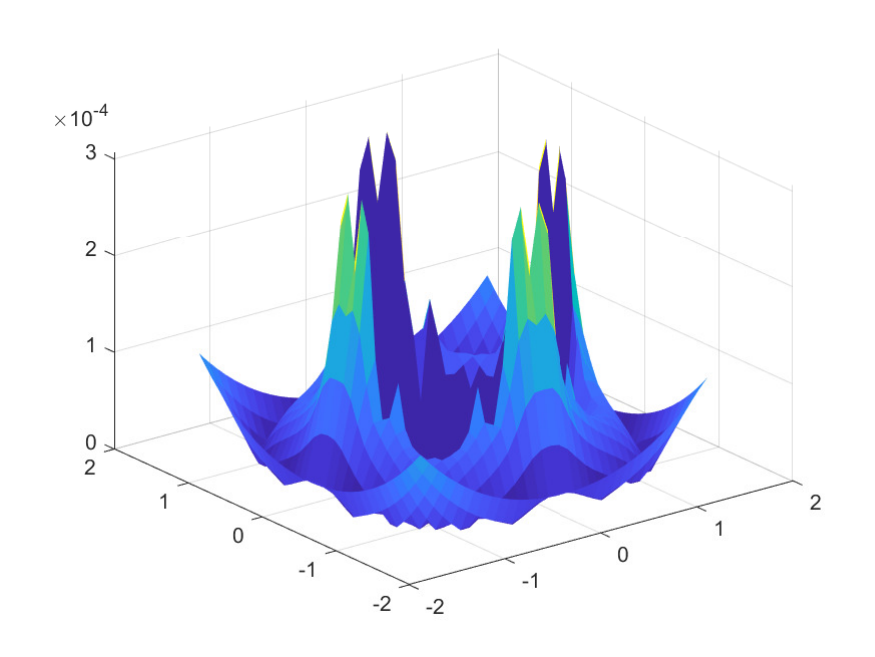}
        \caption{Error between expectation of exact  solution and MC solution}
    \end{subfigure}
    \caption{Absolute error in the expectation: comparison between the SG and MC solutions with respect to the exact solution.}
	\label{fig2}
\end{figure}

Figure~\ref{fig1} demonstrates the high accuracy of the SG method, which also exhibits superior computational efficiency, completing in just 10.49 seconds compared to the MC method's  344.56 seconds for processing $2^{15}$ samples. However, as shown in Figure \ref{fig2}, while the MC solution outperforms the SG solution, it does so at the cost of increased computational resource consumption.

\begin{example}\label{exm2}
	Let $ D=(-1,1)^2$  and 
	\[a(x_1,x_2, \omega) = 1,
	\]
\[
	{f}(x_1,x_2, \omega) = \begin{cases}
		8r^2(x_1^2+x_2^2-1-r^2)*(\exp(\xi_1)+2\exp(\xi_2)),& x_1^2+x_2^2 \leq r^2,\\
		-8(2x_1^2+2x_2^2-r^2)*(\exp(\xi_1)+2\exp(\xi_2)), & x_1^2+x_2^2>r^2,
	\end{cases}
	\]
for $\xi_1,\xi_2 \thicksim \text{i.i.d.}\, \mathrm{U}(- {1}, {1})$ , $g = 0$, and $r = 0.7$.  The exact solution is 
\[
	u(x,y,\omega) = \begin{cases}
		0, & x_1^2+x_2^2 \leq r^2,\\
		(x_1^2+x_2^2-r^2)^2*(\exp(\xi_1)+2\exp(\xi_2)),& x_1^2+x_2^2>r^2.
	\end{cases} 
\]
\end{example}

\begin{table}[!h]
	\centering
	\caption{Relative errors of expected values with $s = (e - e^{-1})/2^3 $}
	\begin{tabular}{ |c|c|c|c|c| } 
		\hline
		$h$		  & $e_{L^2}^{rel,1}(u,u^{hs})$ & order & $e_{H^1}^{rel,1}(u,u^{hs})$ & order \\ \hline
		$1/2^{2}$ &1.3864e-01 &       &3.2220e-01&       \\ \hline
		$1/2^{3}$ &3.5350e-02 &1.9716 &1.6350e-01&0.9786 \\ \hline
		$1/2^{4}$ &8.9860e-03 &1.9760 &8.2092e-02&0.9940 \\ \hline
		$1/2^{5}$ &2.2315e-03 &2.0096 &4.1091e-02&0.9984 \\ \hline
	   \end{tabular}
	   \label{table4}
\end{table}
\begin{table}[!h]
	\centering
	\caption{Relative errors of second moments with $s = (e - e^{-1})/2^3 $}
	\begin{tabular}{ |c|c|c|c|c|} 
		\hline
		$h$		  & $e_{L^2}^{rel,2}(u,u^{hs})$ & order   &  $e_{H^1}^{rel,2}(u,u^{hs})$  & order   \\  \hline
		$1/2^{2}$ &4.1182e-01&		&6.0197e-01&\\ \hline
		$1/2^{3}$ &1.0932e-01&1.9134 &3.1815e-01&0.9200 \\ \hline
		$1/2^{4}$ &2.7766e-02&1.9772 &1.6141e-01&0.9789 \\ \hline
		$1/2^{5}$ &6.9689e-03&1.9944 &8.1007e-02 &0.9947 \\ \hline
	   \end{tabular}  
	   \label{table5} 
\end{table}

\begin{table}[!h]
	\centering
	\caption{Relative errors of expected values with  $h = s/(e-e^{-1})$}
	\begin{tabular}{ |c|c|c|c|c|} 
		\hline
		$s$		  & $e_{L^2}^{rel,1}(u,u^{hs})$ & order & $e_{H^1}^{rel,1}(u,u^{hs})$ & order \\ \hline
		$(e-e^{-1})/2^{1}$ &5.4709e-01 &		 &6.0731e-01 &        \\ \hline
		$(e-e^{-1})/2^{2}$ &1.3864e-01 &1.9804&3.2220e-01 &0.9145  \\ \hline
		$(e-e^{-1})/2^{3}$ &3.5350e-02 &1.9716&1.6350e-01 &0.9786  \\ \hline
		$(e-e^{-1})/2^{4}$ &8.9860e-03 &1.9760&8.2092e-02 &0.9940\\ \hline
	   \end{tabular}
	   \label{table6}
\end{table}         
From Tables~\ref{table4} and~\ref{table5}, we observe that the convergence orders in the $H^1$-norm are approximately 1, which aligns well with the theoretical result given in Theorem~\ref{thm4.3}.

\begin{figure}[htbp]
    \centering
    \begin{subfigure}{0.32\textwidth}
        \includegraphics[width=\linewidth]{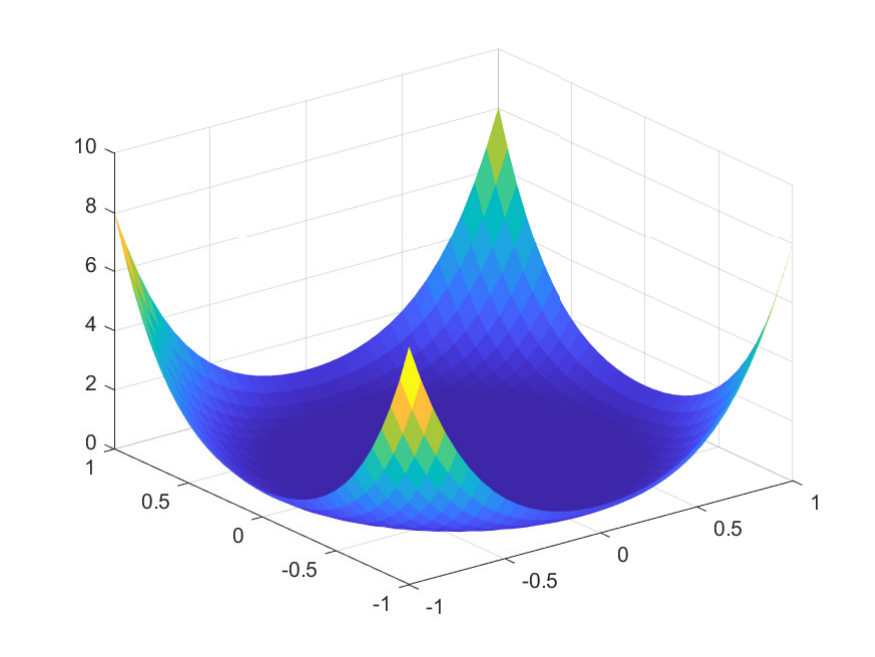}
        \caption{Expectation of the Exact solution }
    \end{subfigure}
    \hfill
    \begin{subfigure}{0.32\textwidth}
        \includegraphics[width=\linewidth]{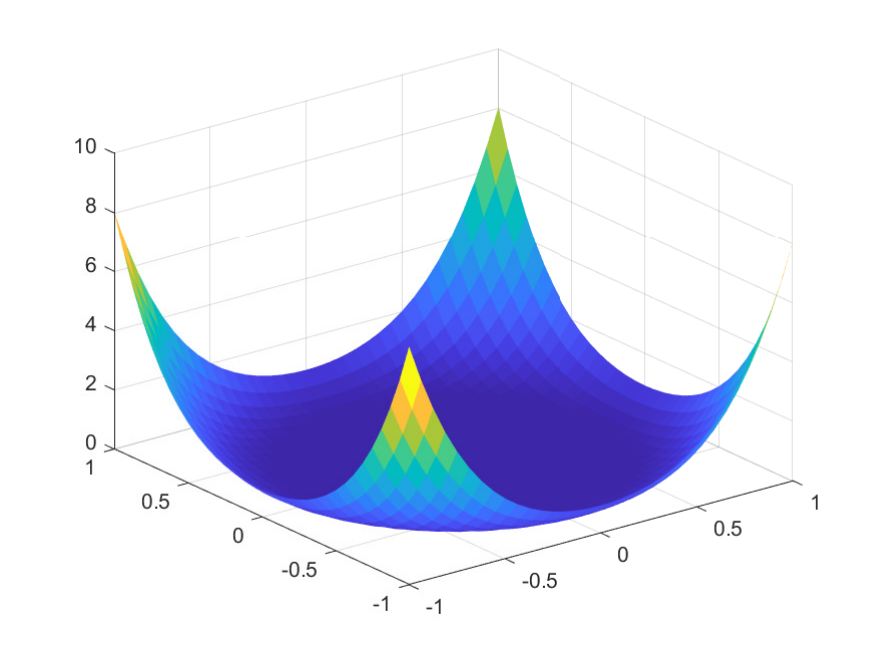}
        \caption{ Expectation of the SG solution}
    \end{subfigure}
    \hfill
    \begin{subfigure}{0.32\textwidth}
        \includegraphics[width=\linewidth]{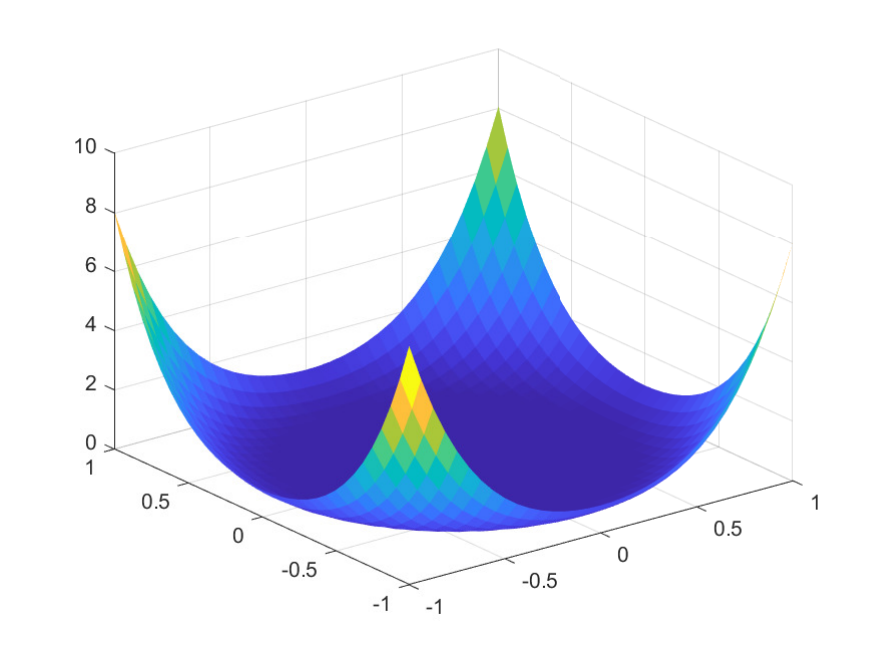}
        \caption{Mean of the MC solution \quad \quad \qquad \quad \quad \qquad   }
    \end{subfigure}

    \vspace{1em}

    \begin{subfigure}{0.32\textwidth}
        \includegraphics[width=\linewidth]{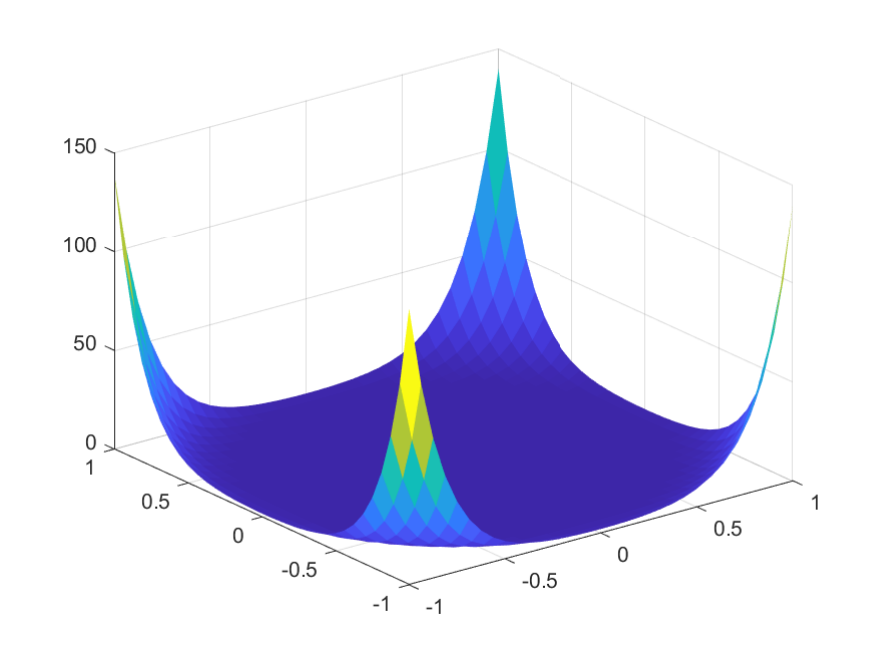}
        \caption{ Variance of the Exact solution  }
    \end{subfigure}
    \hfill
    \begin{subfigure}{0.32\textwidth}
        \includegraphics[width=\linewidth]{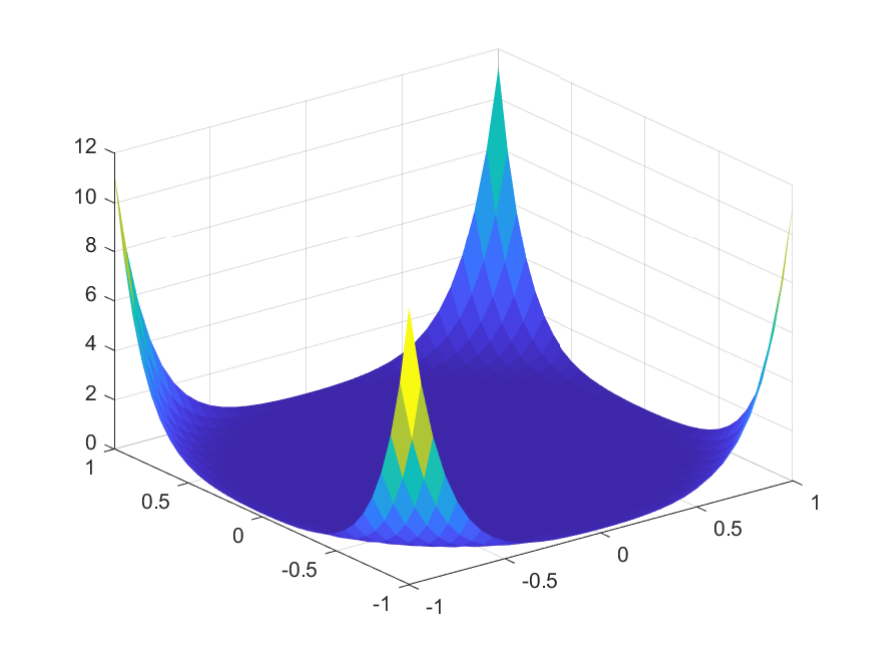}
        \caption{Variance of the SG solution  }
    \end{subfigure}
    \hfill
    \begin{subfigure}{0.32\textwidth}
        \includegraphics[width=\linewidth]{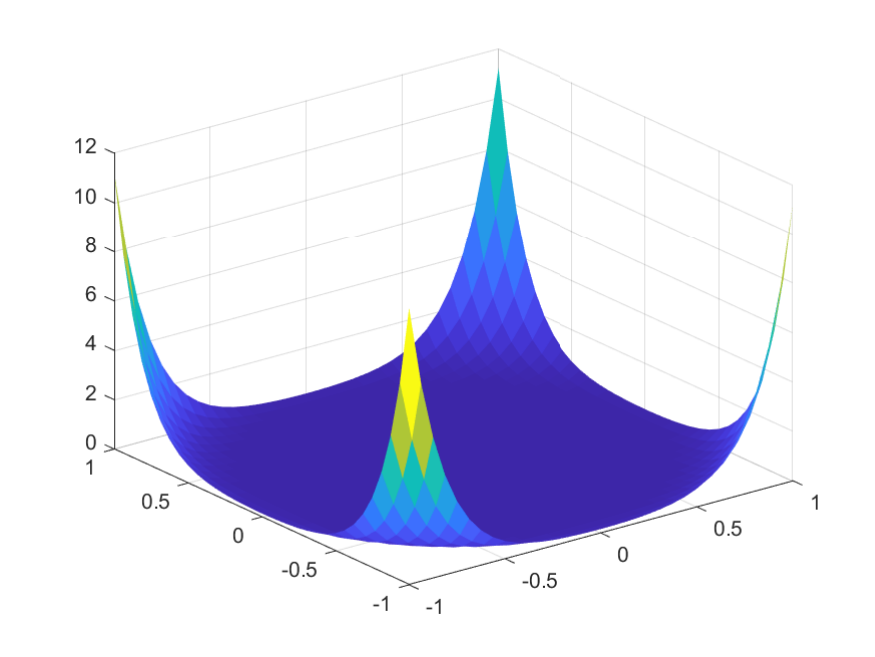}
        \caption{Variance of the MC solution}
    \end{subfigure}

    \caption{Comparison of statistical properties: expectation and variance of the exact solution, SG solution, and MC solution.}
	\label{fig3}
\end{figure}

\begin{figure}[htbp]
    \centering
    \begin{subfigure}{0.45\textwidth}
        \includegraphics[width=\linewidth]{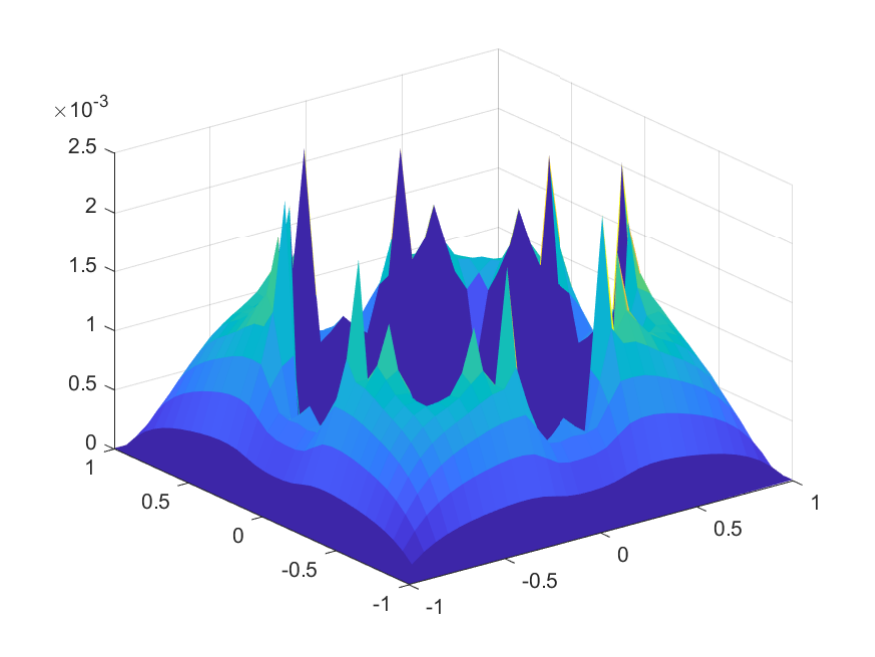}
        \caption{Error between expectation of exact  solution and SG solution}
    \end{subfigure}
    \hfill
    \begin{subfigure}{0.45\textwidth}
        \includegraphics[width=\linewidth]{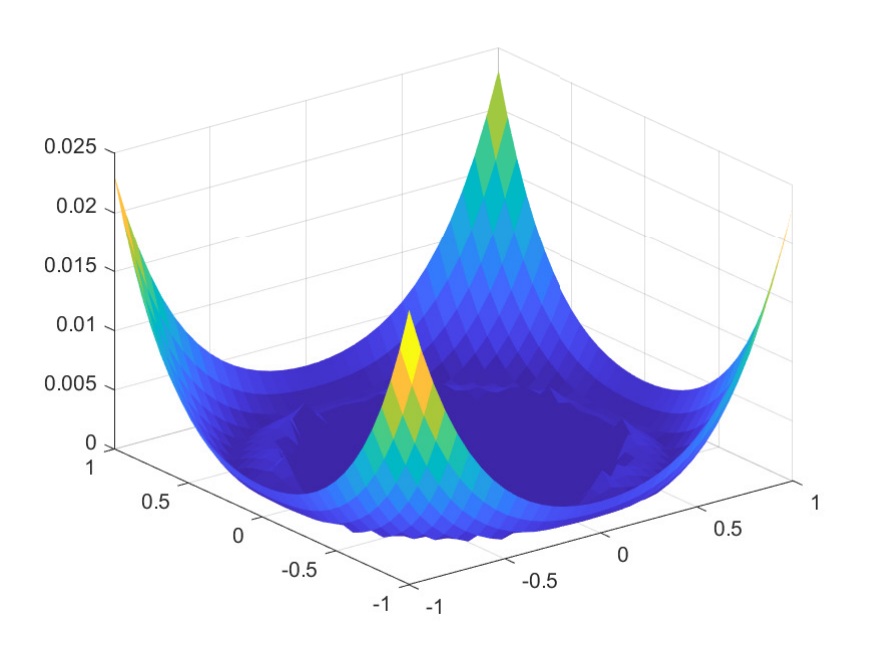}
        \caption{Error between expectation of exact  solution and MC solution}
    \end{subfigure}
    \caption{Absolute error in the expectation: comparison between the SG and MC solutions with respect to the exact solution.}
	\label{fig4}
\end{figure}

The SG method not only achieves high accuracy, as demonstrated in Figures~\ref{fig3} and~\ref{fig4}, but also exhibits superior computational efficiency, completing in just  9.82 seconds. In contrast, the MC method requires 349.29 seconds to process $2^{15}$ samples, yet it still falls short of the true solution. Consequently, the MC method would need significantly more samples to approximate the true solution, thereby increasing its computational demands.

Our numerical experiments demonstrate that both the expectation error and the second moment error converge at an \(O(h)\) rate in the \(H^1\) -norm. 

\section{Conclusion}

In this study, we developed stochastic Galerkin formulations to solve stochastic obstacle problems and established the well-posedness of the discrete formulation. Our numerical analysis demonstrated the effectiveness of the SG method, providing an optimal error estimate for the linear element in the \(H^1\)-norm within the physical space. Numerical experiments validated these theoretical findings, showing that the expectation error and second moment error of the solution converge at a rate of \(O(h)\) in the \(H^1\)-norm, aligning with the derived estimates.

{This approach can be extended to more complex stochastic problems, including time-dependent variational inequalities, while adaptive strategies will be designed to refine the stochastic discretization dynamically based on error indicators, focusing computational effort on regions of high variability. We will also explore the comparative advantages of stochastic collocation methods against the SG approach in handling diverse randomness and solution structures. To address the challenges of high-dimensional random spaces, sparse grids, low-rank tensor approximations, and reduced-order modeling will be employed to mitigate the curse of dimensionality and enhance efficiency. Finally, to improve accuracy in problems with limited regularity, we aim to develop specialized finite element spaces or enriched basis functions that better capture solution features.}

\end{document}